\documentclass[submission,copyright,creativecommons]{eptcs}
\providecommand{\event}{Compositionality} 
\usepackage[utf8]{inputenc}
\usepackage{etex}

\usepackage{xargs}  
\usepackage[pdftex,dvipsnames]{xcolor}  

\usepackage[colorinlistoftodos,prependcaption,textsize=tiny]{todonotes}
\newcommandx{\unsure}[2][1=]{\todo[linecolor=red,backgroundcolor=red!25,bordercolor=red,#1]{#2}}
\newcommandx{\change}[2][1=]{\todo[linecolor=blue,backgroundcolor=blue!25,bordercolor=blue,#1]{#2}}
\newcommandx{\info}[2][1=]{\todo[linecolor=Orange,backgroundcolor=Orange!25,bordercolor=Orange,#1]{#2}}
\newcommandx{\improvement}[2][1=]{\todo[linecolor=Plum,backgroundcolor=Plum!25,bordercolor=Plum,#1]{#2}}
\newcommandx{\thiswillnotshow}[2][1=]{\todo[disable,#1]{#2}}

\usepackage{amsmath,amsthm,amssymb,amsfonts,graphicx}
\usepackage{tikz}
\usepackage{proof}
\usepackage{enumerate}
\usepackage{mathtools}
\usepackage{xypic}
\usepackage{lscape}
\usepackage{multicol}
\usepackage{mathtools}

\usepackage{array}   
\newcolumntype{L}{>{$}l<{$}} 

\tikzstyle{strings}=[baseline={([yshift=-.5ex]current bounding box.center)}]

\tikzset{every picture/.append style={scale=.4}, transform shape, strings}

\tikzset{%
symbol/.style={%
draw=none,
every to/.append style={%
edge node={node [sloped, allow upside down, auto=false]{$#1$}}}
}
}

\usetikzlibrary{shapes.geometric}
\usetikzlibrary{patterns}
\usetikzlibrary{fit}
\usetikzlibrary{positioning}
\usetikzlibrary{calc}
\usetikzlibrary{arrows}
\usetikzlibrary{decorations.markings}
\usetikzlibrary{decorations.pathreplacing}
\usetikzlibrary{shapes}

\pgfdeclarelayer{nodelayer}
\pgfdeclarelayer{edgelayer}
\pgfsetlayers{edgelayer,nodelayer,main}

\tikzset{simple/.style={}}
\tikzset{nothing/.style={outer sep=-3.4pt}}
\tikzset{map/.style={draw,fill=white, rectangle}}

\tikzset{dot/.style={thick, fill=black, circle, scale=1, inner sep = .05cm}}

\tikzset{oplus/.style={draw, scale=0.9,minimum height=.1cm,circle,append after command={
[shorten >=\pgflinewidth, shorten <=\pgflinewidth,]
(\tikzlastnode.north) edge (\tikzlastnode.south)
(\tikzlastnode.east) edge (\tikzlastnode.west)
}
}
}

\tikzset{otimes/.style={draw, rotate=45,scale=0.9,minimum height=.1cm,circle,append after command={
[shorten >=\pgflinewidth, shorten <=\pgflinewidth,]
(\tikzlastnode.north) edge (\tikzlastnode.south)
(\tikzlastnode.east) edge (\tikzlastnode.west)
}
}
}

\tikzset{oa/.style={draw, scale=0.9,minimum height=.1cm,circle,append after command={
[shorten >=\pgflinewidth, shorten <=\pgflinewidth,]
(\tikzlastnode.north) edge (\tikzlastnode.south)
(\tikzlastnode.east) edge (\tikzlastnode.west)
}
}
}

\tikzset{ox/.style={draw, rotate=45,scale=0.9,minimum height=.1cm,circle,append after command={
[shorten >=\pgflinewidth, shorten <=\pgflinewidth,]
(\tikzlastnode.north) edge (\tikzlastnode.south)
(\tikzlastnode.east) edge (\tikzlastnode.west)
}
}
}

\tikzset{circ/.style={
shape=circle, inner sep=1pt, draw}}

\tikzset{fanin/.style={
draw,
shape border rotate=30,
regular polygon,
regular polygon sides=3,
fill=white,
inner sep = .1cm
}
}

\tikzset{fanout/.style={
draw,
shape border rotate=-30,
regular polygon,
regular polygon sides=3,
fill=white,
inner sep = .1cm
}
}

\tikzset{onein/.style={
draw,
shape border rotate=30,
regular polygon,
regular polygon sides=3,
fill=black,
inner sep = .04cm,
scale=1.2
}
}

\tikzset{oneout/.style={
draw,
shape border rotate=-30,
regular polygon,
regular polygon sides=3,
fill=black,
inner sep = .04cm,
scale=1.2
}
}

\tikzset{zeroin/.style={
draw,
shape border rotate=30,
regular polygon,
regular polygon sides=3,
fill=white,
inner sep = .04cm,
scale=1.2
}
}

\tikzset{zeroout/.style={
draw,
shape border rotate=-30,
regular polygon,
regular polygon sides=3,
fill=white,
inner sep = .04cm,
scale=1.2
}
}

\newcommand{\envmap}{
\begin{tikzpicture}[scale=1.5]
\draw (0,0.25) -- (0,-0.05);
\draw (-0.15,-0.05) -- (0.15,-0.05);
\draw (-0.10,-0.1) -- (0.10,-0.1);
\draw (-0.05,-0.15) -- (0.05,-0.15);
\end{tikzpicture}
}

\newcommand{\anotherenvmap}[1]{
\begin{tikzpicture}[scale=#1]
	\begin{pgfonlayer}{nodelayer}
		\node [style=none] (0) at (-2, 0.7) {};
		\node [style=none] (1) at (-2, 0.5) {};
		\node [style=none] (2) at (-2.12, 0.5) {};
		\node [style=none] (3) at (-1.88, 0.5) {};
	\end{pgfonlayer}
	\begin{pgfonlayer}{edgelayer}
		\draw [style=none] (0.center) to (1.center);
		\draw [style=none, bend right=90, looseness=1.75] (2.center) to (3.center);
		\draw [style=none] (2.center) to (3.center);
	\end{pgfonlayer}
\end{tikzpicture}
}

\tikzstyle{none}=[inner sep=-1pt]
\tikzstyle{circle}=[shape=circle,draw]

\tikzset{wires/.style={}}

\tikzset{box/.style={inner sep=0pt, thick, draw=black, text height=1.5ex, text depth=.25ex, text centered, minimum height=3em, anchor=center}}

\renewcommand{\amalg}{\mathbin{\rotatebox[origin=c]{0}{$\prod$}}}
\newcommand{\invamalg}{\mathbin{\rotatebox[origin=c]{180}{$\amalg$}}}

\newcommand{\D}{\mathbb{D}}
\newcommand{\C}{\mathbb{C}}
\newcommand{\V}{\mathbb{V}}
\newcommand{\U}{\mathbb{U}}
\newcommand{\X}{\mathbb{X}}

\newcommand{\Y}{\mathbb{Y}}
\newcommand{\R}{\mathbb{R}}

\newcommand{\m}{{\sf m}}

\newcommand{\CP}{\mathsf{CP}}
\newcommand{\ox}{\otimes}

\newcommand{\oa}{\oplus}
\newcommand{\op}{\mathsf{op}}

\newcommand{\mx}{\mathsf{mx}}

\newcommand{\Core}{\mathsf{Core}}
\newcommand{\FMat}{\mathsf{FMat}}
\newcommand{\Mat}{\mathsf{Mat}}

\newcommand{\dashvv}{\dashv \!\!\!\! \dashv}

\renewcommand{\phi}{\varphi}
\renewcommand{\epsilon}{\varepsilon}

\newtheorem*{lemma*}{Lemma} 
\newtheorem*{theorem*}{Theorem} 

\newtheorem{observation}{Remark}[section]
\newtheorem{lemma}[observation]{Lemma}  
\newtheorem{theorem}[observation]{Theorem}
\newtheorem{definition}[observation]{Definition}
\newtheorem{example}[observation]{Example}

\newtheorem{proposition}[observation]{Proposition} 
\newtheorem{corollary}[observation]{Corollary} 

\makeatletter

\newdimen\w@dth

\def\setw@dth#1#2{\setbox\z@\hbox{\scriptsize $#1$}\w@dth=\wd\z@
\setbox\@ne\hbox{\scriptsize $#2$}\ifnum\w@dth<\wd\@ne \w@dth=\wd\@ne \fi
\advance\w@dth by 1.2em}

\def\t@^#1_#2{\allowbreak\def\n@one{#1}\def\n@two{#2}\mathrel
{\setw@dth{#1}{#2}
\mathop{\hbox to \w@dth{\rightarrowfill}}\limits
\ifx\n@one\empty\else ^{\box\z@}\fi
\ifx\n@two\empty\else _{\box\@ne}\fi}}
\def\t@@^#1{\@ifnextchar_ {\t@^{#1}}{\t@^{#1}_{}}}

\def\t@left^#1_#2{\def\n@one{#1}\def\n@two{#2}\mathrel{\setw@dth{#1}{#2}
\mathop{\hbox to \w@dth{\leftarrowfill}}\limits
\ifx\n@one\empty\else ^{\box\z@}\fi
\ifx\n@two\empty\else _{\box\@ne}\fi}}
\def\t@@left^#1{\@ifnextchar_ {\t@left^{#1}}{\t@left^{#1}_{}}}

\def\two@^#1_#2{\def\n@one{#1}\def\n@two{#2}\mathrel{\setw@dth{#1}{#2}
\mathop{\vcenter{\hbox to \w@dth{\rightarrowfill}\kern-1.7ex
                 \hbox to \w@dth{\rightarrowfill}}%
       }\limits
\ifx\n@one\empty\else ^{\box\z@}\fi
\ifx\n@two\empty\else _{\box\@ne}\fi}}
\def\tw@@^#1{\@ifnextchar_ {\two@^{#1}}{\two@^{#1}_{}}}

\def\tofr@^#1_#2{\def\n@one{#1}\def\n@two{#2}\mathrel{\setw@dth{#1}{#2}
\mathop{\vcenter{\hbox to \w@dth{\rightarrowfill}\kern-1.7ex
                 \hbox to \w@dth{\leftarrowfill}}%
       }\limits
\ifx\n@one\empty\else ^{\box\z@}\fi
\ifx\n@two\empty\else _{\box\@ne}\fi}}
\def\t@fr@^#1{\@ifnextchar_ {\tofr@^{#1}}{\tofr@^{#1}_{}}}

\newdimen\W@dth
\def\setW@dth#1#2{\setbox\z@\hbox{$#1$}\W@dth=\wd\z@
\setbox\@ne\hbox{$#2$}\ifnum\W@dth<\wd\@ne \W@dth=\wd\@ne \fi
\advance\W@dth by 1.2em}

\def\T@^#1_#2{\allowbreak\def\N@one{#1}\def\N@two{#2}\mathrel
{\setW@dth{#1}{#2}
\mathop{\hbox to \W@dth{\rightarrowfill}}\limits
\ifx\N@one\empty\else ^{\box\z@}\fi
\ifx\N@two\empty\else _{\box\@ne}\fi}}
\def\T@@^#1{\@ifnextchar_ {\T@^{#1}}{\T@^{#1}_{}}}

\def\T@left^#1_#2{\def\N@one{#1}\def\N@two{#2}\mathrel{\setW@dth{#1}{#2}
\mathop{\hbox to \W@dth{\leftarrowfill}}\limits
\ifx\N@one\empty\else ^{\box\z@}\fi
\ifx\N@two\empty\else _{\box\@ne}\fi}}
\def\T@@left^#1{\@ifnextchar_ {\T@left^{#1}}{\T@left^{#1}_{}}}

\def\Tofr@^#1_#2{\def\N@one{#1}\def\N@two{#2}\mathrel{\setW@dth{#1}{#2}
\mathop{\vcenter{\hbox to \W@dth{\rightarrowfill}\kern-1.7ex
                 \hbox to \W@dth{\leftarrowfill}}%
       }\limits
\ifx\N@one\empty\else ^{\box\z@}\fi
\ifx\N@two\empty\else _{\box\@ne}\fi}}
\def\T@fr@^#1{\@ifnextchar_ {\Tofr@^{#1}}{\Tofr@^{#1}_{}}}

\def\Two@^#1_#2{\def\N@one{#1}\def\N@two{#2}\mathrel{\setW@dth{#1}{#2}
\mathop{\vcenter{\hbox to \W@dth{\rightarrowfill}\kern-1.7ex
                 \hbox to \W@dth{\rightarrowfill}}%
       }\limits
\ifx\N@one\empty\else ^{\box\z@}\fi
\ifx\N@two\empty\else _{\box\@ne}\fi}}
\def\Tw@@^#1{\@ifnextchar_ {\Two@^{#1}}{\Two@^{#1}_{}}}

\def\to{\@ifnextchar^ {\t@@}{\t@@^{}}}
\def\from{\@ifnextchar^ {\t@@left}{\t@@left^{}}}
\def\tofro{\@ifnextchar^ {\t@fr@}{\t@fr@^{}}}
\def\To{\@ifnextchar^ {\T@@}{\T@@^{}}}
\def\From{\@ifnextchar^ {\T@@left}{\T@@left^{}}}
\def\Two{\@ifnextchar^ {\Tw@@}{\Tw@@^{}}}
\def\Tofro{\@ifnextchar^ {\T@fr@}{\T@fr@^{}}}

\makeatother

\begin{document}

\title{Complete Positivity for Mixed Unitary Categories}
\author{Robin Cockett
\institute{Department of Computer Science}
\institute{University of Calgary\\
Alberta, Canada}
\email{robin@ucalgary.ca}
\and
Priyaa Varshinee Srinivasan
\institute{University of Calgary\\
Alberta, Canada\\ National Institute of Standards and Technology,\\
Maryland, USA}
\email{priyaavarshinee@gmail.com}
}

\def\titlerunning{Complete Positivity for Mixed Unitary Categories (MUCs)}
\def\authorrunning{R.Cockett, P. Srinivasan}

\maketitle 

\begin{abstract}
Coecke and Heunen described completely positive maps in dagger monoidal categories and the {\sf CP}-infinity construction on these categories in order to construct a category of arbitrary dimensional quantum processes. This article generalizes the $\CP$-infinity construction of dagger monoidal categories to mixed unitary categories. Mixed unitary categories, on the one hand, generalize the (compact) dagger monoidal categories, and on the other hand, accommodate arbitrary dimensional quantum processes, both without sacrificing the notion of dual objects. This means that the $\CP$-infinity construction for mixed unitary categories provides a suitable semantics for higher-order quantum programming languages which employ arbitrary dimensional structures.  

The existing results for the $\CP$-infinity construction are shown to generalize to the new setting. In particular, the notion of environment structures generalizes to mixed unitary categories and it is shown that the $\CP$-infinity construction for mixed unitary categories is characterized by this generalized environment structure. 
\end{abstract}

\section{Introduction} 

Categorical quantum mechanics (CQM) uses Hilbert spaces as the {\em de facto} model 
for quantum systems. The category of finite 
dimensional Hilbert spaces and linear maps, which are dagger compact closed, are 
used as a diagrammatic framework to describe and reason about quantum processes. A well-known 
limitation of such compact closed categories is that they model finite dimensional 
Hilbert spaces, but they cannot model infinite dimensional spaces \cite{Heu08}.
Due to this restriction, practical applications of CQM, so far, have been 
focused on the areas such as quantum computing and quantum information theory 
which are off-shoots of finite dimensional quantum mechanics. 

The success of CQM in the study of finite dimensional processes has inspired researchers to extend CQM frameworks to include infinite dimensional processes \cite{CH16, Go16, AH10, RSF17, HeR18}.
The most common approach has been to work within a dagger monoidal setting since the category of arbitrary dimensional Hilbert spaces is dagger monoidal (but not compact closed). 
However, such an approach is not quite satisfactory since, in these categories, the objects need not have duals (a feature which allows wires to be ``bent'' in the graphical calculus).

In \cite{CCS18}, the authors and Comfort addressed this dimensionality issue by generalizing the framework of dagger monoidal categories to mixed unitary categories (MUCs) which are based on linearly distributive categories and $*$-autonomous categories, see Section \ref{Sec: MUCs} for more details. In order to define MUCs, one starts by defining dagger functor for linearly distributive categories. The dagger functor abstracts the notion of unitary evolution of quantum systems by capturing the notion of adjoints. The dagger functor is identity-on-objects for monoidal categories, which forces the objects in a dagger monoidal category to be self-adjoint. However, this is not the case in linearly distributive categories. Indeed, describing the dagger functor for LDCs leads one to a description of unitary isomorphisms for unitary (or self-adjoint) objects. This also leads one to the isomix version of dagger LDCs (a setting in which truth and false are isomorphic). By collecting the unitary objects of a dagger isomix category, one gets a unitary category, which is equivalent to a dagger monoidal category but sitting inside the larger setting of dagger isomix categories: this is the setting of a mixed unitary category. See Figure \ref{fig: MUC schematic} schematic of the mixed unitary category: 

\begin{figure}[h]
\centering
\begin{tikzpicture} [scale=1.8]
	\begin{pgfonlayer}{nodelayer}
		\node [style=circle, scale=14] (0) at (4.5, 0) {};
		\node [style=none] (1) at (-3, 2) {};
		\node [style=none] (2) at (-5, 0) {};
		\node [style=none] (3) at (-3, -2) {};
		\node [style=none] (4) at (-1, 0) {};
		\node [style=none] (5) at (-3, -0.75) {$A \to^{\varphi_A}_{\simeq} A^\dagger$};
		\node [style=none] (6) at (-3, 0.5) {Unitary};
		\node [style=none] (7) at (-3, 0) {category};
		\node [style=none] (8) at (-0.75, 0) {};
		\node [style=none] (9) at (3, 0) {};
		\node [style=none] (10) at (0.25, 0.25) {$\dagger$-isomix};
		\node [style=none] (11) at (0.25, -0.25) {functor};
		\node [style=none] (12) at (4.5, 2) {$\dagger$-isomix};
		\node [style=none] (13) at (4.5, 1.5) {category};
		\node [style=none] (14) at (4, -2) {$B$};
		\node [style=none] (15) at (5.5, -2) {$B^\dagger$};
		\node [style=none] (16) at (4, 0.75) {};
		\node [style=none] (17) at (3.25, 0) {};
		\node [style=none] (18) at (3.5, -1) {};
		\node [style=none] (19) at (4.25, -0.25) {};
		\node [style=none] (20) at (2.75, 1.25) {};
		\node [style=none] (21) at (4.75, 1.25) {};
		\node [style=none] (22) at (4.75, -1.25) {};
		\node [style=none] (23) at (2.75, -1.25) {};
		\node [style=none] (24) at (3.25, 1) {$\Core$};
	\end{pgfonlayer}
	\begin{pgfonlayer}{edgelayer}
		\draw (1.center) to (2.center);
		\draw (2.center) to (3.center);
		\draw (3.center) to (4.center);
		\draw (4.center) to (1.center);
		\draw [->] (8.center) to (9.center);
		\draw [dotted] (16.center) to (17.center);
		\draw [dotted] (17.center) to (18.center);
		\draw [dotted] (18.center) to (19.center);
		\draw [dotted] (19.center) to (16.center);
		\draw (20.center) to (21.center);
		\draw (21.center) to (22.center);
		\draw (22.center) to (23.center);
		\draw (23.center) to (20.center);
	\end{pgfonlayer}
\end{tikzpicture} 
\caption{Schematic of mixed unitary categories}
\label{fig: MUC schematic}
\end{figure}

Mixed unitary categories allow the presence dual objects without the restriction of finite dimensionality.  While every dagger monoidal category is certainly a mixed unitary category the reverse is not the case. Examples of MUCs include Chu spaces \cite{Bar06} and finiteness spaces \cite{Ehr05} which have been used to model infinite dimensional systems in physics and in computer science \cite{ Abr12, Abr13, KiU19, Pra95, Sri21}.  {\em The aim of this article is to describe quantum processes in the general setting of mixed unitary categories. }

In a dagger compact closed category, a (finite dimensional) quantum process is given by a map of the shape shown in Figure \ref{fig: CP map}-(a) below. Such a map is a called a completely positive map because it is an abstract representation of the Kraus decomposition of completely positive linear maps between $\C^\star$ algebras \cite{Cho75}. In order to 
describe arbitrary dimensional quantum processes diagrammatically, Coecke and Heunen  dropped the requirement of 
{\em compactness} and moved to the setting of dagger symmetric monoidal categories ($\dagger$-SMCs). Since, wires cannot be ``bent'' in general in a dagger symmetric monoidal category, they straightened the environment wire $E$ in diagram (a) below to obtaining completely positive maps with a hole as shown in Figure \ref{fig: CP map}-(b). In this manner Coecke and Heunen provided the ${\sf CP}$-infinity construction which would turn any $\dagger$-SMC into one with completely positive maps, and, furthermore, they axiomatized the construction using environment structures.  

\begin{figure}[h]
\centering
 (a)~~ $f$ = \begin{tikzpicture}[scale=1.8]
	\begin{pgfonlayer}{nodelayer}
		\node [style=circle, scale=2.5] (0) at (-1, 3) {};
		\node [style=circle, scale=2.5] (1) at (1, 3) {};
		\node [style=none] (2) at (1, 3) {$g$};
		\node [style=none] (3) at (-1, 3) {$g^{\dag *}$};
		\node [style=none] (4) at (-1.75, 2) {};
		\node [style=none] (5) at (-0.5, 2) {};
		\node [style=none] (6) at (0.5, 2) {};
		\node [style=none] (7) at (1.75, 2) {};
		\node [style=none] (8) at (-1, 4) {};
		\node [style=none] (9) at (1, 4) {};
		\node [style=none] (10) at (-1.25, 3.75) {$A^*$};
		\node [style=none] (11) at (1.25, 3.75) {$A$};
		\node [style=none] (12) at (-2, 1.75) {$B^*$};
		\node [style=none] (13) at (0.25, 2.75) {$E$};
		\node [style=none] (14) at (1.75, 1.25) {};
		\node [style=none] (15) at (2, 1.75) {$B$};
		\node [style=none] (16) at (-1.75, 1.25) {};
	\end{pgfonlayer}
	\begin{pgfonlayer}{edgelayer}
		\draw [in=90, out=-150] (1) to (6.center);
		\draw [bend left] (1) to (7.center);
		\draw [in=90, out=-30] (0) to (5.center);
		\draw [bend right] (0) to (4.center);
		\draw (8.center) to (0);
		\draw (9.center) to (1);
		\draw [bend right=90, looseness=1.75] (5.center) to (6.center);
		\draw (7.center) to (14.center);
		\draw (16.center) to (4.center);
	\end{pgfonlayer}
\end{tikzpicture} 
 ~~~~~~~~ (b)~~ $f$ = ~ \begin{tikzpicture}[scale=1.8]
		\begin{pgfonlayer}{nodelayer}
			\node [style=circle, scale=2] (0) at (1, -0.25) {};
			\node [style=circle, scale=2] (1) at (1, 3.25) {};
			\node [style=none] (2) at (1, 3.25) {$g$};
			\node [style=none] (3) at (1, -0.25) {$g^\dag$};
			\node [style=none] (4) at (1.75, 0.75) {};
			\node [style=none] (5) at (0.25, 0.75) {};
			\node [style=none] (6) at (0.25, 2.25) {};
			\node [style=none] (7) at (1.75, 2.25) {};
			\node [style=none] (8) at (1, -1.25) {};
			\node [style=none] (9) at (1, 4.25) {};
			\node [style=none] (10) at (0.75, -1) {$A$};
			\node [style=none] (11) at (0.75, 4) {$A$};
			\node [style=none] (13) at (0, 1.5) {$E$};
			\node [style=none] (15) at (2, 0.25) {$B$};
			\node [style=none] (16) at (2, 2.75) {$B$};
			\node [style=circle, scale=3, dashed] (17) at (1.75, 1.5) {};
		\end{pgfonlayer}
		\begin{pgfonlayer}{edgelayer}
			\draw [in=90, out=-150] (1) to (6.center);
			\draw [in=90, out=-30, looseness=1.25] (1) to (7.center);
			\draw [in=-90, out=150] (0) to (5.center);
			\draw [in=270, out=30, looseness=1.25] (0) to (4.center);
			\draw (8.center) to (0);
			\draw (9.center) to (1);
			\draw (6.center) to (5.center);
		\end{pgfonlayer}
	\end{tikzpicture} 
 \caption{(a) A completely positive map in a dagger compact closed category; (b) A completely positive map in a dagger monoidal category}
 \label{fig: CP map}
\end{figure}

This article describes completely positive maps for MUCs, and generalizes the $\CP$-infinity construction from dagger monoidal categories to MUCs, developing an axiomatic description of this construction using generalized environment structures. This generalization is not straight forward because unlike in dagger monoidal and dagger-compact closed categories, objects are not in general self-adjoint in a MUC, that is, for an arbitrary object $A$ in the category, $A \neq A^\dagger$. This means a map $g$ cannot be composed with its dagger as shown in Figure \ref{fig: CP map}-(b) along the environment wire $E$ (because, $E$ is not necessarily equal to $E^\dagger$). This problem can be resolved by requiring the object $E$ to be {\em unitary} (meaning $E \simeq E^\dagger$,  see Section \ref{Appendix A. Unitary}).  This allows a formulation of completely positive maps in MUCs, see Section \ref{Sec: Kraus maps} for details. The {\sf CP}-infinity construction for MUCs is discussed in Section \ref{Sec: CP-inf}. The {\sf CP}-infinity construction on a MUC does not, in general, result in a MUC unless every object has a dual. The structures which are inherited by the construction are discussed in Section  \ref{Sec: CP-inf structures}. The proofs in this article extensively use the graphical calculus for MUCs developed in Section \ref{Sec: String calculus}. 

The more sophisticated type system of MUCs becomes important when axiomatizing the {\sf CP}-infinity construction in terms of the environment structures. A dagger monoidal category is said to have an environment structure when every object in the category is equipped with a discarding map $d: A \to I$ (loosely, information held by an object can be discarded freely): this must behave coherently with the dagger monoidal structure. An environment structure for a MUC, on the other hand demands such discarding maps only for the unitary objects. The definition of an environment structure for MUCs and how they are used to axiomatize the {\sf CP}-infinity construction is discussed in Section \ref{Sec: Env maps} and uses the notion of dagger linear functors developed in Section \ref{Sec: Dagger linear functor}. Notably, when a MUC is a dagger monoidal category, one faithfully recovers Coecke and Heunen's {\sf CP}-infinity construction from the {\sf CP}-infinity construction of MUCs. 

A motivation for describing completely positive maps in MUCs is that it may provide a suitable semantics for higher-order quantum programming languages. Selinger and Valiron \cite{Sev08} described a model of a higher-order linear quantum programming language based on the CPM (completely positive maps) construction on the category of finite dimensional Hilbert spaces and linear maps. However, the existing models \cite{Sel04, ClP19, CPD19, PSV14, Sev08} fall short in satisfactorily handling the preservation of the total probability of higher order computations due to the infinite sums involved.  Moreover, infinite datatypes and recursion is handled by augmenting the CPM category with additional properties like infinite biproducts \cite{PSV14}. Having a description of completely positive maps in MUCs enables an investigation of richer models of higher-order quantum computation without some of the drawback listed above.

\vspace{1em}

Note: {\em This article assumes some familiarity with linearly distributive categories (LDCs), mix categories, linear adjoints, linearly distributive functors and linear transformations \cite{CS97a},  $*$-autonomous categories, and dagger compact closed categories \cite{Sel07}. We recall the definition of these structures in Appendix \ref{sec: Appendix A}.  A review of these structures is also available in  \cite[Sec. 2]{CCS18}.}

\section{Preliminaries: Mixed unitary categories}
\label{Sec: MUCs}

Mixed unitary categories draw inspiration from the proof theory of multiplicative linear logic, with categorical semantics in linearly distributive categories. In \cite{CCS18, Sri21} it is shown how to coherently add a dagger ($\dagger$) to the doctrine of linearly distributive categories (and $*$-autonomous categories) to obtain the proof theory of $\dagger$-linear logic. 

Linearly distributive categories are categories with two monoidal structures called the tensor, $\ox$, and the par $\oa$ linked by natural transformations called linear distributors: 

\[ \partial^L: A \ox (B \oa C) \to (A \ox B) \oa C ~~~~~~~~ \partial^R: (A \oa B) \ox C \to A \oa (B \ox C) \]

The distributor is referred to as linear because, unlike in the usual distribution of product over sums, $A \times (B + C) \simeq (A \times B) + (A \times C)$, in the linear distribution, type $A$ is not duplicated.

In order to model quantum processes it is necessary, firstly, to specialize to the ``mix'' case of linear logic in where there is a coherent transformation from the tensor to the par of two objects ${\sf mx}: A \ox B \to A \oa B$ (and an isomorphism ${\sf m}: \top \to \bot$ from the tensor unit to the par unit). 

In order to model quantum structures one needs a notion of ``unitary structure''.  This is determined by objects $U$ which are in the core (that is for which ${\sf mx}: U \ox X \to U \oa X$ is an isomorphism) and which are equipped with a (coherent) isomorphism $\varphi: U \to U^\dagger$.  The subcategory of unitary objects form a $\dagger$-monoidal category which sits inside a larger category of quantum processes.  Intuitively one may think of the unitary objects as being the finite dimensional objects of the setting and thus potentially forming a standard model of CQM.  From this perspective, a mixed unitary category consists of a standard model of CQM sitting in a larger category which models the proof theory of $\dagger$-linear logic and includes infinite dimensional processes.

\subsection{Dagger linearly distributive categories}

Conventionally, in categorical quantum mechanics a dagger is defined as a contravariant
endofunctor which is stationary on objects ($A^\dagger = A$) and an involution ($f^{\dagger \dagger} = f$). However,
in an LDC where there are two tensor products interacting by linear distribution ($\partial^L: A \ox (B \oa C) \to (A \ox B) \oa C$), and the dagger must minimally flip the tensor products to 
maintain the directionality of the distributor maps. This because, if the dagger were stationary on objects it would provide a map $\partial^\dagger: (A \ox B)\oa C \to A\ox (B \oa C)$: this is not a valid map an LDC. Hence, unlike for dagger monoidal categories, a dagger for an LDC cannot be stationary on objects.  However, it is still possible for it to be an involution. 

\begin{definition}
\label{Definition: daggerLDC elaborate}
A dagger linearly distributive category is an LDC with a functor $(\_)^\dag:\X^\op\to \X$ and 
natural isomorphisms 
\begin{align*}
\text{ \bf laxors: }  A^\dag \ox B^\dag &\xrightarrow{ \lambda_\ox} (A\oa B)^\dag ~~~~~ A^\dag \oa B^\dag \xrightarrow{ \lambda_\oa} (A\ox B)^\dag \\
\top &\xrightarrow{\lambda_\top} \bot^\dag ~~~~~~~~~~~~~~~~~~~~~ \bot \xrightarrow{\lambda_\bot} \top^\dag \\
\text{ \bf involutor: } & A \xrightarrow{\iota} (A^\dag)^\dag 
\end{align*}
such that certain coherence conditions hold \cite{CCS18}.
\end{definition}

A {\bf symmetric \dag-LDC} is a $\dagger$-LDC which is also symmetric LDC and the symmetry isomorphisms behave coherently with the laxors:
\[
\begin{tabular}{cc}
\xymatrix{
A^\dag \ox B^\dag                          \ar@{->}[r]^{\lambda_\ox} \ar@{->}[d]_{c_\ox}         
  & (A\oa B)^\dag                          \ar@{->}[d]^{c_\oa^\dag}\\
B^\dag \ox A^\dag                          \ar@{->}[r]_{\lambda_\ox}
  & (B\oa A)^\dag\\
} & \xymatrix{
A^\dag \oa  B^\dag                          \ar@{->}[r]^{\lambda_\oa} \ar@{->}[d]_{c_\oa}  
  & (A\ox B)^\dag                          \ar@{->}[d]^{c_\ox^\dag}\\
B^\dag \oa A^\dag                          \ar@{->}[r]_{\lambda_\oa}
  & (B\ox A)^\dag\\
}
\end{tabular}
\]

\begin{definition}
A {\bf $\dagger$-mix category} is a $\dagger$-LDC which is also a mix category, where the mix map $\m: \bot \to \top$ additionally satisfies the following commuting diagram. A {\bf $\dagger$-isomix category} is a $\dagger$-mix category which is also an isomix category. 

\[ \mbox{\bf [$\dagger$-\text{mix}]}  ~~~~\begin{array}[c]{c} 
\xymatrix{
\bot                 \ar@{->}[r]^{{\sf m}} \ar@{->}[d]_{\lambda_\bot}  
  & \top             \ar@{->}[d]^{\lambda_\top}\\
\top^\dag            \ar@{->}[r]_{{\sf m}^\dag}
  & \bot^\dag
} 
\end{array} \]
\end{definition}

\begin{lemma}
\label{lemma:mixdagger}
Suppose $\X$ is a $\dagger$-mix category then the following diagram commutes:
\[
\xymatrix{ A^\dag \ox B^\dag  \ar[r]^{\mx} \ar[d]_{\lambda_\ox}&  A^\dag \oa B^\dag \ar[d]^{\lambda_\oa} \\
                (A \oa B)^\dag \ar[r]_{\mx^\dag}  & (A \ox B)^\dag }
\]
\end{lemma}

\begin{lemma}
In a $\dagger$-LDC, if $A$ is left linear dual to $B$, then $B^\dagger$ is left linear dual to $A^\dagger$.
\end{lemma}
\begin{proof}
    Suppose $(\eta, \epsilon): A \dashvv B$ then $(\lambda_\top\epsilon^\dag\lambda_\oa^{-1},\lambda_\ox\eta^\dagger \lambda_\bot^{-1}):B^\dagger \dashvv A^\dagger$.
\end{proof}


\subsection{Unitary isomorphisms in dagger LDCs}
\label{Appendix A. Unitary}

The notion of unitary maps is central to both quantum information theory as well as quantum mechanics since the evolution of a closed quantum system is described by such maps. Categorically, within a $\dagger$-category, a unitary map is an isomorphism $f: A \to B$ such that $f^{-1} =  f^\dagger$. This definiton of unitary isomorphism cannot be used directly within the framework of $\dagger$-LDCs since the types of $f^{-1}: B \to A$ and  $f^\dagger: B^\dagger \to A^\dagger$ are different. It is therefore
apparent that one can only ask to have unitary isomorphisms  between certain objects, which we call ``unitary objects'':

\begin{definition} \cite[Defn. 5.1]{CCS18}
	\label{defn: unitary structure}
A  $\dagger$-isomix category, $\X$ has {\bf unitary structure} in case there is an essentially small class of objects $\mathcal{U}$, called the {\bf unitary objects} of $\X$ such that
\begin{enumerate}[{\bf [U.1]}]
\item for all $A \in \mathcal{U}$, $A \in  \Core(\X)$, and $A$ is equipped with an isomorphism, $\varphi_A: A \to A^\dag$, called the {\bf unitary structure map} of $A$
\item $\mathcal{U}$ is closed under $(\_)^\dag$ so that for all $A \in \mathcal{U}$, $\varphi_{A^\dag} = ((\varphi_A)^{-1})^\dag$ 
\item for all $A \in \mathcal{U}$, the following diagram commutes:
 \[   \xymatrix{  A   \ar[d]_{\varphi_A} \ar[drrr]^{\iota}  & \\ A^\dag \ar[rrr]_{\varphi_{A^\dag}}  & & & (A^\dag)^\dag  } \]
\item $\bot, \top \in \mathcal{U}$ satisfy:
\[ \xymatrixcolsep{2pc}
\xymatrix{
\bot \ar[r]^{\varphi_\bot} \ar[d]_{\lambda_\bot} \ar[dr]^{\m} & \bot^\dagger \ar[d]^{\lambda_\top^{-1}}  \\
\top^\dagger \ar[r]_{\varphi_\top^{-1}} & \top
}
\]
\item If $A , B \in \mathcal{U}$, then $A \ox B$ and $A \oa B \in \mathcal{U}$ satisfy:
\[ (a) ~~~~~ \xymatrixcolsep{3pc}
\xymatrix{
A \ox B \ar[r]^{\varphi_A \ox \varphi_B}_{\simeq} \ar@/_2pc/[rrr]_{\mx}&
 A^\dagger \ox B^\dagger \ar[r]^{\lambda_\oa}_{\simeq} & 
 (A \oa B) ^\dagger \ar[r]^{\varphi_{A \oa B}^{-1}} _{\simeq} &
A \oa B
}
\]
\[ (b) ~~~~~ \xymatrixcolsep{3pc}
\xymatrix{
A \ox B \ar[r]^{\varphi_{A \ox B}}_{\simeq} \ar@/_2pc/[rrr]_{\mx}&
 (A \ox B)^\dagger \ar[r]^{\lambda_\ox^{-1}}_{\simeq} & 
 A^\dagger \oa B^\dagger \ar[r]^{\varphi_A^{-1} \oa \varphi_B^{-1}} _{\simeq} &
A \oa B
}
\]
 \end{enumerate}
\end{definition}

\begin{definition}
A {\bf unitary category} is a $\dagger$-isomix category $\U$ in which every object is unitary.
\end{definition}

We can now define what it means for a isomorphism to be unitary:

\begin{definition}
Suppose $A$ and $B$ are unitary objects. An isomorphism $A\xrightarrow{f} B$ is said to be a {\bf unitary isomorphism} if the following diagram commutes:
\[  \xymatrix{A   \ar[r]^{\varphi_A}    \ar[d]_{f} \ar[r]^{\varphi_A} & A^\dag \\ B  \ar[r]_{\varphi_B} & B^\dag  \ar[u]_{f^\dag}  }  \]
\end{definition}

Observe that $\varphi$ is ``twisted'' natural for all unitary isomorphisms, thus, unitary isomorphisms compose and contain the identity maps. In a category in which the unitary structure maps are identity morphisms, one recovers the usual notion of unitary isomorphisms.

Often we shall want the unitary objects to have linear adjoints (or duals) but we shall need the analogue of dagger duals from categorical quantum mechanics \cite[Defn. 2.6]{Sel07}:

\begin{definition} \label{unitary-duals}
A {\bf unitary linear duality} $(\eta, \epsilon): A \dashvv_{~u} B$ between unitary objects  $A$ and $B$ is a linear duality satisfying in addition:
\[
\xymatrix{ \\
{\bf [Udual.]} \\
}~~~
\xymatrix{
\top \ar@{}[ddrr]|{(a)} \ar[rr]^{\eta} \ar[d]_{\lambda_\top}  & & A \oa B \ar[d]^{\varphi_A \oa \varphi_B} \\
\bot^\dagger \ar[d]_{\epsilon^\dag} & & A^\dagger \oa B^\dagger \ar[d]^{c_\oa} \\ 
(B \ox A)^\dag \ar[rr]_{\lambda_\oa^{-1}} & & B^\dagger \oa A^\dagger} 
~~~~~~~~~~~
\xymatrix{
A \ox B \ar@{}[ddrr]|{(b)} \ar[rr]^{\varphi_A \ox \varphi_B} \ar[d]_{c_\ox} & & A^\dag \ox B^\dag \ar[d]^{\lambda_\ox} \\
B \ox A \ar[d]_{\epsilon} & & (A \oa B)^\dagger \ar[d]^{\eta^\dagger} \\
\bot \ar[rr]_{\lambda_\bot} & & \top^\dagger } 
\]
\end{definition}

In a $\dagger$-monoidal monoidal category, the above described unitary duality precisely coincides with the notion of dagger duals.

\begin{lemma}
Suppose $(\eta_1, \epsilon_1): V_1 \dashvv_{~u} U_1$ and $(\eta_2, \epsilon_2): V_2 \dashvv_{~u} U_2$. Then, $(V_1 \otimes V_2) \dashvv_{~u} (U_1 \oplus U_2)$.
\end{lemma}
\begin{proof}
Define $(\eta', \epsilon'): (V_1 \otimes V_2) \dashvv_{~u} (U_1 \oplus U_2)$ where 
$\eta' = \begin{tikzpicture} 
	\begin{pgfonlayer}{nodelayer}
		\node [style=circle] (0) at (-4, 3) {$\eta_1$};
		\node [style=circle] (1) at (-2, 3) {$\eta_2$};
		\node [style=otimes] (2) at (-4, 1.75) {};
		\node [style=oplus] (3) at (-2, 1.75) {};
		\node [style=none] (4) at (-4, 1) {};
		\node [style=none] (5) at (-2, 1) {};
	\end{pgfonlayer}
	\begin{pgfonlayer}{edgelayer}
		\draw [style=none, in=15, out=-165, looseness=1.00] (1) to (2);
		\draw [style=none, bend left, looseness=1.25] (1) to (3);
		\draw [style=none, in=180, out=-15, looseness=1.00] (0) to (3);
		\draw [style=none, bend left=45, looseness=1.25] (2) to (0);
		\draw [style=none] (2) to (4.center);
		\draw [style=none] (3) to (5.center);
	\end{pgfonlayer}
\end{tikzpicture} ~~~~~~~ 
\epsilon' = \begin{tikzpicture} 
	\begin{pgfonlayer}{nodelayer}
		\node [style=circle] (0) at (-4, 1) {$\epsilon_1$};
		\node [style=circle] (1) at (-2, 1) {$\epsilon_2$};
		\node [style=oplus] (2) at (-4, 2.25) {};
		\node [style=otimes] (3) at (-2, 2.25) {};
		\node [style=none] (4) at (-4, 3) {};
		\node [style=none] (5) at (-2, 3) {};
	\end{pgfonlayer}
	\begin{pgfonlayer}{edgelayer}
		\draw [style=none, in=-15, out=165, looseness=1.00] (1) to (2);
		\draw [style=none, bend right, looseness=1.25] (1) to (3);
		\draw [style=none, in=180, out=15, looseness=1.00] (0) to (3);
		\draw [style=none, bend right=45, looseness=1.25] (2) to (0);
		\draw [style=none] (2) to (4.center);
		\draw [style=none] (3) to (5.center);
	\end{pgfonlayer}
\end{tikzpicture}
$
then this is easily checked to be a unitary linear adjoint.
\end{proof}

\subsection{Mixed unitary categories}

With the definition of unitary objects in place, one could ask for a $\dagger$-isomix category in which all the objects are unitary; or a $\dagger$-isomix category with a full sub $\dagger$-isomix category of unitary objects. The former notion is formalized by so called unitary categories, generalising $\dagger$-monoidal categories; the latter is formalized by so called  mixed unitary categories.

A unitary category is a compact LDC by defintion.  A $\dagger$-monoidal category is a strict unitary category in which the unitary structure map and the mix map are identity morphisms. Similarily, a $\dagger$-compact closed category is a strict unitary category in which all objects have unitary duals.

A mixed unitary category (MUC) can be schematically represented as follows. Informally, a MUC is a unitary category that `lives' within the core of a dagger-isomix category. Recall that for any object $U$ within the core and for any other object A, $U \ox A \simeq U \ox B$, and $U \ox A \simeq U \ox B$.  See Figure \ref{fig: MUC schematic} for a schematic representation of mixed unitary categories. 

The formal definition of MUCs is given below:

\begin{definition}
A {\bf mixed unitary category} (MUC) is a $\dagger$-isomix category, $\C$, equipped with a strong $\dagger$-isomix functor 
$M: \U \to \C$ from a unitary category $\U$ to $\C$ such that there exists the following natural transformations:
\[ \mx': M(U) \oa X \to M(U) \ox X  \text{ with } \mx  ~\mx' = 1 \text{ and }\mx' ~ \mx = 1 \]
\[ \mx'': X \oa M(U) \to X \ox M(U) \text{ with } \mx  ~\mx'' = 1 \text{ and }\mx'' ~ \mx = 1 \]
A mixed unitary category, $M: \U \to \C$ is {\bf symmetric} if the functor $M$, the 
unitary category $\U$, and the $\dagger$-isomix category $\C$ are symmetric. If  $M: \U \to \C$ is an MUC where  $\C$ is a $*$-autonomous category, then $ M: \U \to \C$ is said to be a {\bf $*$-MUC}. If moreover every object in $\U$ has unitary duals, $M: \U \to \C$ is said to be a {\bf $*$-mixed unitary category with duals} ($*$-MUdC).
\end{definition}

In the definition of a MUC, the requirement of a transformation $\mx'$ which is inverse to $\mx$ ensures that the functor 
$M: \U \to \C$ factors through $\Core(\C)$. 

All dagger monoidal categories together with the identity functor are mixed unitary categories.  The inclusion of a full subcategory of unitary objects of a $\dagger$-isomix category with unitary structure is a MUC.

 A functor between two MUCs $(\U \to^M \C)$ and $(\V \to^N \Y)$ is a square commuting upto $\dagger$-linear natural isomorphism $\gamma$:
 \[ \xymatrix{ \U \ar[d]_{F'} \ar@{}[drr]|{\Downarrow~\gamma} \ar[rr]^M & & \X \ar[d]^F \\ \V \ar[rr]_{N} & & \Y} \] $F$ and $F'$ are $\dagger$-isomix functors and $F'$ preserves unitary structure. See Section \ref{Sec: Dagger linear functor} for the definition of a $\dagger$-linear natural isomorphism and conditions for preservation of unitary structure. 

\subsection{Examples of MUCs}

An important source of examples of MUCs is given by considering an LDC and the inclusion functor of its core. Thus, the category of finite dimensional framed vector spaces and linear maps, finiteness modules, and the category of Chu spaces of complex vector spaces over $\X$  give MUCs in this manner \cite{CCS18}. Any $\dagger$-monoidal category is also a MUC with $M=id$ and the unitary structure given by identity maps. In fact, every unitary category is equivalent to a $\dagger$-monoidal category: in a unitary category, however, the $\dagger$ is not assumed to be stationary on objects but merely isomorphic via the unitary structure map.

In this article, we shall exemplify our constructions on the following two examples of MUCs:

\begin{example}~

{\em
\begin{enumerate}[(1)]
\item $\R^* \subset \C$: consider the discrete monoidal category, $\C$, of complex numbers defined as follows:

\vspace{-0.25cm}
\begin{description}
\item[Objects:] $a+ib \in \C$
\item[Maps:]  Identity maps only $c = c$
\item[Tensor:] multiplication is given by $(a+ib) \ox (x+iy) := (ax - by) + i (ay + bx)$ with the unit 1
\item[Dagger:] $(a+ib)^\dagger := a - ib$
\end{description}
$\C$ is a compact LDC (that is $\ox \simeq \oa$)  with a non-stationary dagger functor.

The subcategory $\R^*$ of non-zero real numbers is a unitary category with the unitary structure map given by identity maps. Thus, $\R \subset \C$ is a mixed unitary category.

\item $\Mat_\C \subset \FMat_\C$: The second example uses finiteness spaces \cite{Ehr05}. A {\bf finiteness space}, $(X, \mathcal{A}, \mathcal{B})$, consists of a set $X$ and subsets $\mathcal{A}, \mathcal{B} \subseteq P(X)$ such that $\mathcal{B} = \mathcal{A}^\perp$, in the sense that 
\[ \mathcal{B} = \{ b | b \in P(X) \text{ with for all } a \in A, |a \cap b| < \infty \}, \]
and $\mathcal{A} = \mathcal{B}^\perp$.
 
 A {\bf finiteness relation}, $(X, \mathcal{A}, \mathcal{B}) \to^{R} (Y, \mathcal{A}', \mathcal{B}')$ is relation $X \to^{R} Y$ such that 
 \[ \forall A \in \mathcal{A}. AR \in \mathcal{A}'  ~~\mbox{and}~~\forall B' \in \mathcal{B}'. RB' \in \mathcal{B} \]
 
 Finiteness spaces with finiteness relation form a $*$-autonomous category \cite{Ehr05}. One can define a category of finiteness matrices, $\FMat_\C$ as follows:
 
 \begin{description}
\item[Objects:] Finiteness spaces $(X, \mathcal{A}, \mathcal{B})$
\item[Maps:] $(X, \mathcal{A}, \mathcal{B}) \to^{M} (Y, \mathcal{A}', \mathcal{B}')$ is a matrix $X \times Y \to^{M} \C$ such that 
\[ supp(M) := \{ (x,y) | x \in X, y \in Y \text{ and } M(x,y) \neq 0 \} \]
is a finiteness relation from $(X, \mathcal{A}, \mathcal{B})$ to $(Y, \mathcal{A}', \mathcal{B}')$.
\item[Dagger:] $(X, \mathcal{A}, \mathcal{B})^\dagger := (X, \mathcal{B}, \mathcal{A})$ and $M^\dagger$ is the complex conjugate of $M$.
\end{description}

The finiteness typing of these matrices ensures that on matrix composition the sums are always finite. ${\sf Mat}_\C$, the category of finite matrices over complex numbers is a full subcategory of $ \FMat_\C$ which is determined by the objects, $(X, P(X), P(X))$, where $X$ is a finite set. ${\sf Mat}_\C$  is a unitary category (with unitary structure given by identity maps), and is a well-known $\dagger$-compact closed category. The inclusion ${\sf Mat}_\C \subset \FMat_\C$ is a mixed unitary category.
\end{enumerate}
} 
 \end{example}

\section{Dagger linear functors}
\label{Sec: Dagger linear functor}

(For definition of linear functors, see Appendix \ref{Sec: Appendix B}.)

Clearly the functors and transformations between $\dagger$-LDCs must ``preserve'' the dagger in some sense.  Precisely we have:

\begin{definition}
$F: \X \to \Y$ is a {\bf $\dagger$-linear functor} between $\dagger$-LDCs when it is a linear functor equipped with a linear natural isomorphism 
$\rho^F= (\rho_\ox^F: F_\ox(A^\dagger) \to F_\oa(A)^\dagger ,\rho_\oa^F:  F_\ox(A)^\dagger \to F_\oa(A^\dagger))$ called the {\bf preservator}, 
such that  the following diagrams commute:
\[ 
\xymatrix{
F_\ox(X) \ar[r]^{\iota} \ar[d]_{F_\ox(\iota)} \ar@{}[dr]|{\mbox{\tiny {\bf [$\dagger$-LF.1]}}} & 
F_\ox(X)^{\dagger \dagger} \ar@{<-}[d]^{(\rho^F_\oa)^\dagger} \\
F_\ox(X^{\dagger \dagger}) \ar[r]_{\rho^F_\ox} & F_\oa(X^\dagger)^\dagger
} ~~~~~~~~~ \xymatrix{
F_\oa(X) \ar[r]^{\iota} \ar[d]_{F_\oa(\iota)}  \ar@{}[dr]|{\mbox{\tiny {\bf [$\dagger$-LF.2]}}} & F_\oa(X)^{\dagger \dagger} \ar[d]^{(\rho^F_\ox)^\dagger} \\
F_\oa(X^{\dagger \dagger}) \ar@{<-}[r]_{\rho^F_\oa} & F_\ox(X^\dagger)^\dagger
}
\]
\end{definition}

Observe that when $F$ is a mix functor between $\dagger$-isomix categories, then  $F_\ox= F_\oa$ and the preservators become pairwise inverses: $\rho^F_\ox = (\rho^F_\oa)^{-1}$.  
This means the squares {\bf [$\dagger$-LF.1]} and {\bf [$\dagger$-LF.2]} coincide to give a single condition for the tensor preservator:
\[ \xymatrix{
F(X) \ar[r]^{\iota} \ar[d]_{F(\iota)} \ar@{}[dr]|{\mbox{\tiny {\bf [$\dagger$-isomix]}}} & 
F(X)^{\dagger \dagger} \ar@{->}[d]^{(\rho^F_\ox)^\dagger} \\
F(X^{\dagger \dagger}) \ar[r]_{\rho^F_\ox} & F(X^\dagger)^\dagger
} \]

For linear natural transformations $\beta: F \to G$ between $\dagger$-linear functors we demand that $\beta_\ox$ and $\beta_\oa$ are related by:
\[ \xymatrix{F_\ox(A^\dagger) \ar[d]_{\rho^F_\ox} \ar[rr]^{\beta_\ox} && G_\ox(A^\dagger)  \ar[d]^{\rho^G_\ox} \\
                   (F_\oa(X))^\dagger \ar[rr]_{\beta_\oa^\dagger} & &  (G_\oa(X))^\dagger}
 ~~~~~~
    \xymatrix{(G_\ox(X))^\dagger \ar[d]_{\rho^G_\oa} \ar[rr]^{\beta_\ox^\dagger} &&  (F_\ox(X))^\dagger \ar[d]^{\rho^F_\oa} \\
                   G_\oa(A^\dagger) \ar[rr]_{\beta_\oa} & &  F_\oa(A^\dagger)} \]
                   We call these {\bf $\dagger$-linear natural transformations}. Notice that this means that $\beta_\ox$ is completely determined by $\beta_\oa$ in the following sense:
 \[ \xymatrix{ F_\ox(A) \ar[d]_{F_\ox(\iota)}  \ar[rr]^{\beta_\ox} && G_\ox(A) \ar[d]^{G_\ox(\iota)} \\
                     F_\ox(A^{\dagger\dagger})  \ar[d]_{\rho^F_\ox} \ar[rr]^{\beta_\ox} & & G_\ox(A^{\dagger\dagger}) \ar[d]^{\rho^G_\ox}  \\
                     F_\oa(A^\dagger)^\dagger \ar[rr]_{\beta_\oa^\dagger} && G_\oa(A^\dagger)^\dagger } \]
 Because the vertical maps are isomorphisms, this diagram can be used to express $\beta_\ox$ in terms of $\beta_\oa$.  Similarly $\beta_\oa$ can be expressed in terms 
 of $\beta_\ox$.  Thus, it is possible to express the coherences in terms of just one of these transformations.

Our next step is to generalize the $\CP^\infty$-construction on dagger monoidal categories to the mixed unitary categories. We first introduce a circuit calculus for mixed unitary categories which is an extension of proof nets of linearly distributive categories and is similar to the pictures used in classical categorical quantum mechanics.

\section{String Calculus for Mixed Unitary Categories}
\label{Sec: String calculus}

To facilitate reasoning within MUCs, it is useful to employ a circuit calculus built on the circuit calculus for LDCs introduced in \cite{BCST96}. The extended circuit calculus for mixed unitary categories includes dagger boxes,  components for unitary structure maps, and inverse mixor morphisms.

\subsection{Dagger functor boxes}

Suppose $\X$ is a $\dagger$-LDC and  $f: A \rightarrow B \in \X$. Then, the map $f^\dagger: B^\dagger \rightarrow A^\dagger$ is graphically depicted as follows:
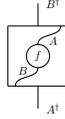
\[ 
 \]
The rectangle is a functor box for the $\dag$-functor.  Notice how we use vertical mirror reflection to express the contravariance of the $\dag$-functor. By functoriality of $(\_)^\dagger$, we have:
%
.

These contravariant functor boxes compose as expected, contravariantly. Given maps $f: A \to B$ and $g: B \to C$:
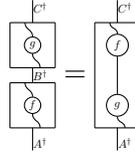
\[
%
\]

The following are the representations of the basic natural isomorphisms of a $\dagger$-LDC:
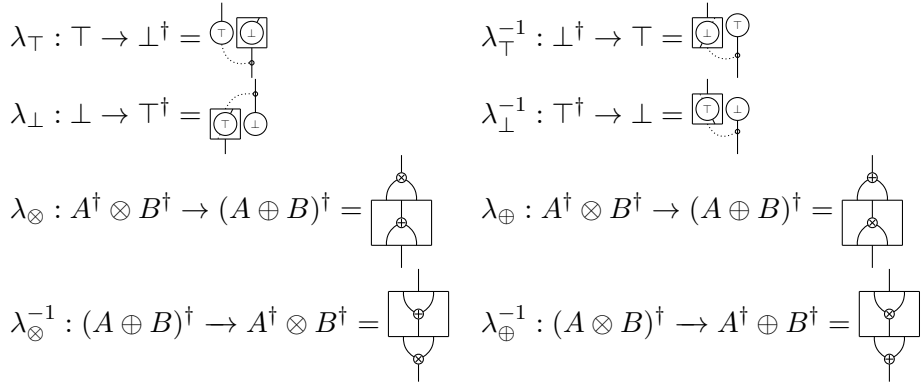
\[
\begin{tabular}{L L}
\lambda_\top: \top \rightarrow \bot^\dagger = 
\begin{tikzpicture}
	\begin{pgfonlayer}{nodelayer}
		\node [style=none] (0) at (-2, 0) {};
		\node [style=none] (1) at (-2, -1) {};
		\node [style=none] (2) at (-1, -1) {};
		\node [style=none] (3) at (-1, 0) {};
		\node [style=circle] (4) at (-1.5, -0.5) {$\bot$};
		\node [style=none] (5) at (-1.25, 0) {};
		\node [style=none] (6) at (-1.75, -1) {};
		\node [style=none] (7) at (-1.5, -2) {};
		\node [style=none] (8) at (-1.5, -1) {};
		\node [style=circle, scale=0.4] (9) at (-1.5, -1.5) {};
		\node [style=circle] (10) at (-2.5, -0.5) {$\top$};
		\node [style=none] (11) at (-2.5, 0.5) {};
	\end{pgfonlayer}
	\begin{pgfonlayer}{edgelayer}
		\draw [style=none] (0.center) to (1.center);
		\draw [style=none] (1.center) to (2.center);
		\draw [style=none] (2.center) to (3.center);
		\draw [style=none] (3.center) to (0.center);
		\draw [style=none] (4) to (5.center);
		\draw [style=none] (7.center) to (8.center);
		\draw [style=none] (11.center) to (10);
		\draw [densely dotted, bend right=45, looseness=1.25] (10) to (9);
	\end{pgfonlayer}
\end{tikzpicture}  & 
\lambda_\top^{-1}: \bot^\dagger \rightarrow \top =
\begin{tikzpicture}
	\begin{pgfonlayer}{nodelayer}
		\node [style=none] (0) at (-2, 0) {};
		\node [style=none] (1) at (-2, -1) {};
		\node [style=none] (2) at (-1, -1) {};
		\node [style=none] (3) at (-1, 0) {};
		\node [style=circle] (4) at (-1.5, -0.5) {$\bot$};
		\node [style=none] (5) at (-1.25, 0) {};
		\node [style=none] (6) at (-1.75, -1) {};
		\node [style=none] (7) at (-1.5, -1) {};
		\node [style=circle] (8) at (-0.5, -0.25) {$\top$};
		\node [style=none] (9) at (-1.5, 0) {};
		\node [style=none] (10) at (-1.5, 0.5) {};
		\node [style=none] (11) at (-0.5, -2) {};
		\node [style=circle, scale=0.4] (12) at (-0.5, -1.25) {};
	\end{pgfonlayer}
	\begin{pgfonlayer}{edgelayer}
		\draw [style=none] (0.center) to (1.center);
		\draw [style=none] (1.center) to (2.center);
		\draw [style=none] (2.center) to (3.center);
		\draw [style=none] (3.center) to (0.center);
		\draw [style=none] (4) to (6.center);
		\draw [style=none] (10.center) to (9.center);
		\draw [style=none] (8) to (11.center);
		\draw [densely dotted, bend right=45, looseness=1.00] (7.center) to (12);
		\draw [densely dotted] (4) to (5.center);
	\end{pgfonlayer}
\end{tikzpicture}  \\
\end{tabular}
\]

Dagger boxes interact with involutor $A \xrightarrow{\iota} A^{\dagger\dagger}$ as follows:
$$
%
$$

In general, one need not have a legal proof net inside a $\dagger$-box: the required correctness criterion is discussed in \cite{MP05}.

\subsection{Unitary structure map}

For any unitary object, its unitary structure isomorphism $A \xrightarrow{\varphi_A} A^\dagger$ is drawn as as a downward pointing triangle:
$$
%
$$

Diagrammatic representation of axioms {\bf [U.5(a)]}, {\bf [U.4(a)]} and {\bf [Udual(a)]} (see Appendix \ref{Appendix A. Unitary}) for unitary structure of a $\dagger$-isomix category are as follows. 

\begin{center}
%
\end{center}

\subsection{Inverse of the mixor}

The mixor (obtained from the mix map $\bot \to^{\m} \top$ as below) $U_1 \ox U_2 \to^{\mx} U_1 \oa U_2$ is an isomorphism whenever one of $U_1$ and $U_2$ is a unitary object:  the inverse of the mix map (when it exists inverse are represented as follows:

$$
\mx_{U_1,U_2}:=
%
$$

Observe that $\mx^{-1}$ maps asscocitated to different unitary objects slide past each other, and indeed by naturality, $\mx^{-1}$ slides over components in circuits:
$$
%
$$

\section{Quantum channels for mixed unitary categories}

In quantum mechanics, the representation of a physical state can be either {\em pure} 
or {\em mixed}. A representation is mixed when it is a statistical ensemble of 
possible states of the system. If the representation is pure, then one knows the exact 
quantum state of the system.  While a pure state is represented as a vector in a 
Hilbert Space, a mixed state is represented as a positive self-adjoint operator on the 
Hilbert Space.  The mixed state formalism of quantum mechanics is very useful in 
practice since in an experimental setting our knowledge about the state of 
a quantum system is often limited. 

In the mixed state formalism, a quantum process is a completely positive map (sending 
positive self-adjoint operators to positive self-adjoint operators) which preserves 
trace. 

The following theorem by Choi characterizes the form of completely positive maps:
\begin{theorem} \cite[Theorem 1]{Cho75} Let $H$ and $K$ be finite dimensional Hilbert spaces.  
	A linear map $\phi: \mathcal{L}(H) \to \mathcal{L}(K)$ is completely 
	positive if and only if there exists a collection of  linear maps, $\{M_i | M_i \in \mathcal{L}(H,K)$ 
	such that for all $A \in \mathcal{L}(H)$, 
	\[ \phi(A) = \sum_i M_i^\dag A M_i \] 
	where $M_i^\dag$ is a adjoint of $M_i$.
\end{theorem}

The equation in the above theorem is referred to as the {\bf Kraus decomposition} of $\phi$, and the 
collection of maps, $\{M_i | M_i \in \mathcal{L}(H,K)$, are referred to as the {\bf Kraus operators}.  

In this section, we discuss completely positive maps in the MUC setting. We begin by describing Kraus maps in MUCs based on the formulation of Kraus maps for dagger monoidal categories \cite{CH16}. 

\subsection{Kraus maps}
\label{Sec: Kraus maps}

\begin{definition} A {\bf Kraus map} $(f,U): A \to B$ in a mixed unitary category, $M: \U \to \C$, is a map $f: A \to M(U) \oa B \in \C$ for some $U \in \U$. $U$ is called the ancillary system of $f$. 
\end{definition}

Note that, while a Kraus map $(f,U)$ can be between arbitrary types, the ancillary system, $U$, is always a unitary object. 

A Kraus map may be connected to its dagger along its ancillary system giving rise to a combinator which acts on ``test maps''. Two Kraus maps are equivalent when their effects on test maps are indistinguishable.

\begin{definition} Given a MUC,  $\U \to^{M} \C$, two Kraus maps $(f,U) , (g,V): A \to B$ are {\bf equivalent}, $(f,U) \sim (g,V)$, if for all {\bf test maps} $h: B \ox C \rightarrow M(W)$, where $C \in C$ and $W \in \U$, the following equality holds:
\begin{align*}
\begin{tikzpicture} 
	\begin{pgfonlayer}{nodelayer}
		\node [style=none] (0) at (-1.75, -3.5) {};
		\node [style=none] (1) at (-1.75, -3.5) {};
		\node [style=none] (2) at (-1.25, -2) {};
		\node [style=none] (3) at (-0.75, -0.25) {};
		\node [style=none] (4) at (-0.75, -3) {};
		\node [style=none] (5) at (-0.5, -3) {};
		\node [style=none] (6) at (-0.25, -2) {};
		\node [style=none] (7) at (-0.25, -5.5) {};
		\node [style=none] (8) at (-0.5, -2) {};
		\node [style=none] (9) at (-0.75, -2) {};
		\node [style=circle] (10) at (-0.75, -2.5) {$h$};
		\node [style=none] (11) at (-1, -3) {};
		\node [style=none] (12) at (-1.25, -3) {};
		\node [style=none] (13) at (-1, -2) {};
		\node [style=none] (14) at (-0.25, -3) {};
		\node [style=none] (15) at (-1, -0) {};
		\node [style=none] (16) at (-0.75, -0.25) {};
		\node [style=none] (17) at (-0.5, -0) {};
		\node [style=none] (18) at (-0.5, 3.25) {};
		\node [style=circle] (19) at (-0.75, 1.25) {$h$};
		\node [style=none] (20) at (-0.75, -0) {};
		\node [style=none] (21) at (-2, 3.25) {};
		\node [style=none] (22) at (-2, -5) {};
		\node [style=none] (23) at (-2, -5.5) {};
		\node [style=circle] (24) at (-2, -4) {$f$};
		\node [style=none] (25) at (-2.5, -3.5) {};
		\node [style=none] (26) at (-2.5, -4.5) {};
		\node [style=none] (27) at (-1.5, -4.5) {};
		\node [style=none] (28) at (-1.5, -3.5) {};
		\node [style=none] (29) at (-1.75, -4.5) {};
		\node [style=none] (30) at (-2.25, -4.5) {};
		\node [style=none] (31) at (-2, -3.5) {};
		\node [style=circle] (32) at (-2, 2.25) {$f$};
		\node [style=none] (33) at (-2.5, -0) {};
		\node [style=none] (34) at (-2.25, -3.5) {};
		\node [style=none] (35) at (-1.75, -3.5) {};
		\node [style=none] (36) at (-2.75, -0) {};
		\node [style=none] (37) at (-2.25, -0) {};
		\node [style=none] (38) at (-2.5, -0.25) {};
		\node [style=none] (39) at (-2.5, -0) {};
		\node [style=none] (40) at (-2, -4.5) {};
		\node [style=none] (41) at (-2.75, 0.75) {};
		\node [style=none] (42) at (-0.5, 0.75) {};
		\node [style=none] (43) at (-0.5, 0.5) {};
		\node [style=none] (44) at (-3, 0.25) {};
		\node [style=none] (45) at (-2, 0.25) {};
		\node [style=none] (46) at (-3, -0.75) {};
		\node [style=none] (47) at (-2, -0.75) {};
		\node [style=none] (48) at (-2.25, -0.5) {$M$};
		\node [style=none] (49) at (-1.25, 0.25) {};
		\node [style=none] (50) at (-0.25, 0.25) {};
		\node [style=none] (51) at (-1.25, -0.75) {};
		\node [style=none] (52) at (-0.25, -0.75) {};
		\node [style=none] (53) at (-0.5, -0.5) {$M$};
		\node [style=circle] (54) at (-2.5, -1.5) {$\rho$};
		\node [style=circle] (55) at (-0.75, -1.25) {$\rho$};
	\end{pgfonlayer}
	\begin{pgfonlayer}{edgelayer}
		\draw [style=none] (0.center) to (1.center);
		\draw [style=none] (13.center) to (10);
		\draw [style=none] (10) to (8.center);
		\draw [style=none] (10) to (4.center);
		\draw [style=none] (2.center) to (6.center);
		\draw [style=none] (6.center) to (14.center);
		\draw [style=none] (14.center) to (12.center);
		\draw [style=none] (12.center) to (2.center);
		\draw [style=none, in=90, out=-90, looseness=1.00] (5.center) to (7.center);
		\draw [style=none] (15.center) to (17.center);
		\draw [style=none] (17.center) to (16.center);
		\draw [style=none] (16.center) to (15.center);
		\draw [style=none, in=-90, out=60, looseness=0.75] (19) to (18.center);
		\draw [style=none] (19) to (20.center);
		\draw [style=none] (22.center) to (23.center);
		\draw (25.center) to (26.center);
		\draw (26.center) to (27.center);
		\draw (28.center) to (27.center);
		\draw (25.center) to (28.center);
		\draw (24) to (31.center);
		\draw (24) to (30.center);
		\draw (24) to (29.center);
		\draw [in=90, out=-135, looseness=0.75] (32) to (33.center);
		\draw (36.center) to (37.center);
		\draw (37.center) to (38.center);
		\draw (38.center) to (36.center);
		\draw (40.center) to (22.center);
		\draw [in=180, out=-45, looseness=1.25] (32) to (19);
		\draw (21.center) to (32);
		\draw [in=-150, out=30, looseness=1.00] (41.center) to (42.center);
		\draw [in=90, out=-90, looseness=1.00] (11.center) to (0.center);
		\draw (44.center) to (46.center);
		\draw (46.center) to (47.center);
		\draw (47.center) to (45.center);
		\draw (45.center) to (44.center);
		\draw (49.center) to (50.center);
		\draw (50.center) to (52.center);
		\draw (52.center) to (51.center);
		\draw (51.center) to (49.center);
		\draw (38.center) to (54);
		\draw (3.center) to (55);
		\draw (55) to (9.center);
		\draw [in=98, out=-90, looseness=1.00] (54) to (34.center);
	\end{pgfonlayer}
\end{tikzpicture} =  \begin{tikzpicture} 
	\begin{pgfonlayer}{nodelayer}
		\node [style=none] (0) at (-1.75, -3.5) {};
		\node [style=none] (1) at (-1.75, -3.5) {};
		\node [style=none] (2) at (-1.25, -2) {};
		\node [style=none] (3) at (-0.75, -0.25) {};
		\node [style=none] (4) at (-0.75, -3) {};
		\node [style=none] (5) at (-0.5, -3) {};
		\node [style=none] (6) at (-0.25, -2) {};
		\node [style=none] (7) at (-0.25, -5.5) {};
		\node [style=none] (8) at (-0.5, -2) {};
		\node [style=none] (9) at (-0.75, -2) {};
		\node [style=circle] (10) at (-0.75, -2.5) {$h$};
		\node [style=none] (11) at (-1, -3) {};
		\node [style=none] (12) at (-1.25, -3) {};
		\node [style=none] (13) at (-1, -2) {};
		\node [style=none] (14) at (-0.25, -3) {};
		\node [style=none] (15) at (-1, -0) {};
		\node [style=none] (16) at (-0.75, -0.25) {};
		\node [style=none] (17) at (-0.5, -0) {};
		\node [style=none] (18) at (-0.5, 3.25) {};
		\node [style=circle] (19) at (-0.75, 1.25) {$h$};
		\node [style=none] (20) at (-0.75, -0) {};
		\node [style=none] (21) at (-2, 3.25) {};
		\node [style=none] (22) at (-2, -5) {};
		\node [style=none] (23) at (-2, -5.5) {};
		\node [style=circle] (24) at (-2, -4) {$g$};
		\node [style=none] (25) at (-2.5, -3.5) {};
		\node [style=none] (26) at (-2.5, -4.5) {};
		\node [style=none] (27) at (-1.5, -4.5) {};
		\node [style=none] (28) at (-1.5, -3.5) {};
		\node [style=none] (29) at (-1.75, -4.5) {};
		\node [style=none] (30) at (-2.25, -4.5) {};
		\node [style=none] (31) at (-2, -3.5) {};
		\node [style=circle] (32) at (-2, 2.25) {$g$};
		\node [style=none] (33) at (-2.5, -0) {};
		\node [style=none] (34) at (-2.25, -3.5) {};
		\node [style=none] (35) at (-1.75, -3.5) {};
		\node [style=none] (36) at (-2.75, -0) {};
		\node [style=none] (37) at (-2.25, -0) {};
		\node [style=none] (38) at (-2.5, -0.25) {};
		\node [style=none] (39) at (-2.5, -0) {};
		\node [style=none] (40) at (-2, -4.5) {};
		\node [style=none] (41) at (-2.75, 0.75) {};
		\node [style=none] (42) at (-0.5, 0.75) {};
		\node [style=none] (43) at (-0.5, 0.5) {};
		\node [style=none] (44) at (-3, 0.25) {};
		\node [style=none] (45) at (-2, 0.25) {};
		\node [style=none] (46) at (-3, -0.75) {};
		\node [style=none] (47) at (-2, -0.75) {};
		\node [style=none] (48) at (-2.25, -0.5) {$M$};
		\node [style=none] (49) at (-1.25, 0.25) {};
		\node [style=none] (50) at (-0.25, 0.25) {};
		\node [style=none] (51) at (-1.25, -0.75) {};
		\node [style=none] (52) at (-0.25, -0.75) {};
		\node [style=none] (53) at (-0.5, -0.5) {$M$};
		\node [style=circle] (54) at (-2.5, -1.5) {$\rho$};
		\node [style=circle] (55) at (-0.75, -1.25) {$\rho$};
	\end{pgfonlayer}
	\begin{pgfonlayer}{edgelayer}
		\draw [style=none] (0.center) to (1.center);
		\draw [style=none] (13.center) to (10);
		\draw [style=none] (10) to (8.center);
		\draw [style=none] (10) to (4.center);
		\draw [style=none] (2.center) to (6.center);
		\draw [style=none] (6.center) to (14.center);
		\draw [style=none] (14.center) to (12.center);
		\draw [style=none] (12.center) to (2.center);
		\draw [style=none, in=90, out=-90, looseness=1.00] (5.center) to (7.center);
		\draw [style=none] (15.center) to (17.center);
		\draw [style=none] (17.center) to (16.center);
		\draw [style=none] (16.center) to (15.center);
		\draw [style=none, in=-90, out=60, looseness=0.75] (19) to (18.center);
		\draw [style=none] (19) to (20.center);
		\draw [style=none] (22.center) to (23.center);
		\draw (25.center) to (26.center);
		\draw (26.center) to (27.center);
		\draw (28.center) to (27.center);
		\draw (25.center) to (28.center);
		\draw (24) to (31.center);
		\draw (24) to (30.center);
		\draw (24) to (29.center);
		\draw [in=90, out=-135, looseness=0.75] (32) to (33.center);
		\draw (36.center) to (37.center);
		\draw (37.center) to (38.center);
		\draw (38.center) to (36.center);
		\draw (40.center) to (22.center);
		\draw [in=180, out=-45, looseness=1.25] (32) to (19);
		\draw (21.center) to (32);
		\draw [in=-150, out=30, looseness=1.00] (41.center) to (42.center);
		\draw [in=90, out=-90, looseness=1.00] (11.center) to (0.center);
		\draw (44.center) to (46.center);
		\draw (46.center) to (47.center);
		\draw (47.center) to (45.center);
		\draw (45.center) to (44.center);
		\draw (49.center) to (50.center);
		\draw (50.center) to (52.center);
		\draw (52.center) to (51.center);
		\draw (51.center) to (49.center);
		\draw (38.center) to (54);
		\draw (3.center) to (55);
		\draw (55) to (9.center);
		\draw [in=98, out=-90, looseness=1.00] (54) to (34.center);
	\end{pgfonlayer}
\end{tikzpicture} 
\end{align*}
\end{definition}

Note that the $\mx^{-1}$ map can be slid up and down along the wires of the ancillary objects $M(U)$ and $M(V)$ by naturality of the $\mx$ map. The diagram includes covariant functor boxes for $M$ and contravariant functor boxes for the dagger. The lefthand diagram is given equationally as follows:
\[
A \ox C \to^{f \ox 1} (X \oa B) \ox C \to^{\delta} X \oa (B \ox C) \to^{1 \oa h} M(U) \oa M(W) \to^{\mx^{-1}} M(U) \ox M(W)\]\[ \to^{M(\varphi_{U}) \ox M(\varphi_V)} M(U^\dagger) \ox M(V^\dagger) \to^{\rho \ox \rho} M(U)^\dagger \ox M(W)^\dagger \to^{1 \ox (h^\dagger \lambda_\oa^{-1})}  M(U)^\dagger \ox (B^\dagger \oa C^\dagger) \] \[\to^{\delta} (M(U)^\dagger \ox B^\dagger) \oa C^\dagger \to^{\lambda_\ox \oa 1} (M(U) \oa B)^\dagger \oa C^\dagger \to^{f^\dagger \oa 1} A^\dagger \oa C^\dagger \to^{\lambda_\oa} (A \ox C)^\dagger
\]

The natural isomorphism $\rho: M(U^\dagger) \to M(U)^\dagger$ is the preservator of the $\dagger$-isomix functor, $M$, which ensures coherence with the $\dagger$ from $\U$ to $\C$.

By forgetting the test maps and gluing Kraus map with its dagger, one gets a notationally convenient combinator which can be diagrammatically represented by:
 \[  
 \]

 An equivalence class of Kraus morphisms is a {\bf quantum channel}. If $(f,U): A \to B$ and $(g, V): A \to B$ are Kraus maps for which $U$ and $V$ are unitarily isomorphic, they are necessarily equivalent with respect to this  relation:

\begin{lemma}
\label{Lemma: unitary equivalence} 
Let $(f,U), (g,V) : A \to B$ be Kraus morphisms. If $~U \to^{\alpha} V$ is a unitary isomorphism and $f( \alpha \oa 1) = g$, then $(f, U) \sim (g, V)$.
\end{lemma}
\begin{proof} Let $h: B \ox C \to M(X)$ be any test map. Then, 
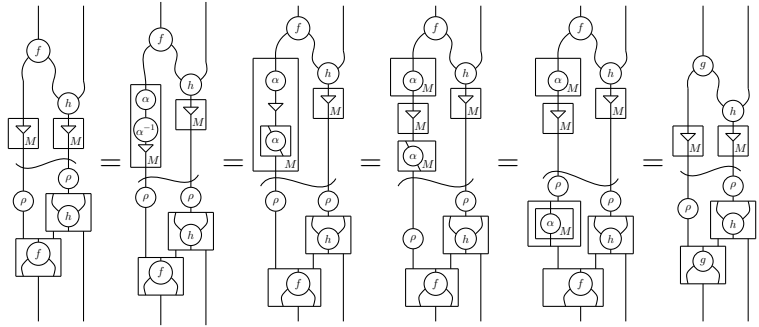
\begin{align*}
%
\end{align*}
\end{proof}

In a $*$-MUC, the equivalence relation on Kraus maps is given equationally as follows:

\begin{lemma}
Suppose $M: \U \to \C$,  is a $*$-MUC, that is every object in $\C$ has  linear dual, then two Kraus maps $(f, U)$ and $(g, V)$ are equivalent if and only if
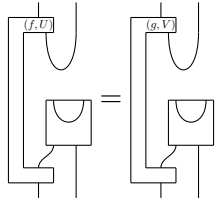
\[
%
 \]
\end{lemma}

Let us now examine Kraus maps in our running examples: 
\begin{example}~

{\em 
\begin{enumerate}[(1)]
\item 
In the MUC, $\R \subset \C$, let $c, c'$ be any two complex numbers. Kraus maps in $\R \subset \C$ are $ (=, r): c \to c' $ such that $c = rc'.$ If $c' \neq 0$, then there is at most one Kraus map $(=,r): c \to c'$. If $c'=0$, then $c=0$ and for all $r' \in \R$, $(=,r) \sim (=,r'): c \to c'$. Thus,  in the complex plane, there are only Kraus maps between those complex numbers that can be connected by a line that extends through the origin making it a $2$-dimensional projective space. 
\[
\begin{tikzpicture} [scale=2]
	\begin{pgfonlayer}{nodelayer}
		\node [style=none] (0) at (0, 2.5) {};
		\node [style=none] (1) at (0, -2.25) {};
		\node [style=none] (2) at (-2.5, -0) {};
		\node [style=none] (3) at (2.5, -0) {};
		\node [style=none] (4) at (1.5, 1.5) {};
		\node [style=none] (5) at (1.5, -1.5) {};
		\node [style=none] (6) at (-1.5, -1.5) {};
		\node [style=none] (7) at (-1.5, 1.5) {};
		\node [style=none] (8) at (2.5, 0.25) {$\R$};
		\node [style=none] (9) at (0.25, 2.5) {$\iota$};
		\node [style=none] (10) at (0.75, 2) {};
		\node [style=none] (11) at (-0.75, -2) {};
	\end{pgfonlayer}
	\begin{pgfonlayer}{edgelayer}
		\draw [<-] (0.center) to (1.center);
		\draw [->](2.center) to (3.center);
		\draw [blue] (7.center) to (5.center);
		\draw [blue] (6.center) to (4.center);
		\draw [blue] (10.center) to (11.center);
	\end{pgfonlayer}
\end{tikzpicture}
\]

\item In $\Mat_\C \subset \FMat_\C$, every Kraus map $(M, \C^n): (X, \mathcal{A}, \mathcal{A}^\perp) \to (Y, \mathcal{B}, \mathcal{B}^\perp)$ is given by the sum of {\bf pure completely positive maps} i.e.,Kraus maps with $\C$ as ancillary object:
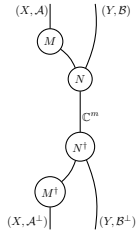
\[
\begin{tikzpicture}
	\begin{pgfonlayer}{nodelayer}
		\node [style=none] (0) at (0, 3.75) {};
		\node [style=circle] (1) at (0, 2.5) {$M$};
		\node [style=none] (2) at (-0.6, 3.5) {$(X, \mathcal{A})$};
		\node [style=circle] (4) at (0, -2.5) {$M^\dagger$};
		\node [style=none] (5) at (0, -3.75) {};
		\node [style=none] (6) at (-0.75, -3.5) {$(X, \mathcal{A}^\perp)$};
		\node [style=circle] (7) at (1, 1.25) {$N$};
		\node [style=none] (8) at (1.5, 3.75) {};
		\node [style=circle] (9) at (1, -1) {$N^\dagger$};
		\node [style=none] (10) at (1.5, -3.75) {};
		\node [style=none] (11) at (2.2, 3.5) {$(Y, \mathcal{B})$};
		\node [style=none] (12) at (1.35, -0) {$\C^m$};
		\node [style=none] (13) at (2.25, -3.5) {$(Y, \mathcal{B}^\perp)$};
		\node [style=none] (14) at (0.75, 2.25) {};
		\node [style=none] (15) at (0.4, -1.75) {};
	\end{pgfonlayer}
	\begin{pgfonlayer}{edgelayer}
		\draw (0.center) to (1);
		\draw (5.center) to (4);
		\draw [bend right=15, looseness=1.00] (7) to (8.center);
		\draw [bend left=15, looseness=1.00] (9) to (10.center);
		\draw (7) to (9);
		\draw [bend left=15, looseness=1.00] (9) to (4);
		\draw [bend left=15, looseness=1.00] (1) to (7);
	\end{pgfonlayer}
\end{tikzpicture} 
 \]

{\bf Choi's theorem} states that every completely positive map can be written as a sum of pure completely positive maps. Analogously, every Kraus map in the category $\FMat_\C$ can be written as a sum of pure maps as above.  Given a Kraus map $(M, \C^m)$, here is the argument:
\begin{align*}
(M \ox 1) (1 \ox NN^\dagger)(M^\dagger \oa 1) &= (M \ox 1) \left( \left( \sum_i \invamalg_i \amalg_i  \right) \ox NN^\dagger \right) (M^\dagger \oa 1) \\
&=  \sum_i \left(  (M(\invamalg_i \ox 1) \ox 1) (1 \ox hh^\dagger) (((\amalg_i \oa 1)M^\dagger) \oa 1)  \right) \\ 
&= \sum_i (M_i \ox 1) (1 \ox NN^\dagger) (M_i^\dagger \oa 1)
\end{align*}
\end{enumerate}
}
\end{example}

\subsection{$\CP$-infinity construction for MUCs}
\label{Sec: CP-inf}

The CPM construction \cite{Sel07} on dagger compact closed categories applied to the concrete category of finite dimensional Hilbert Spaces and linear maps produces a category of mixed states and quantum processes. Coecke and Heunen \cite{CH16} generalized the CPM construction to $\dag$-symmetric monoidal categories, and thus, to infinite dimensions. They call the generalized construction the $\CP$-infinity construction.  In this section, we generalize the $\CP$-infinity construction to MUCs: thus, our construction coincides with the original $\CP$-infinity construction when the MUC is a $\dagger$-monoidal category. 

The  $\CP$-infinity construction for mixed unitary categories is as follows:

\begin{definition}
Given a mixed unitary category, $M: \U \to \C$, define $\CP^\infty(M: \U \to \C)$ to have:
\begin{description}
\item[Objects:] Same as $\C$
\item[Maps:] Equivalence classes of Kraus morphisms from X. That is, $[(f,U)]:A \to B \in  \CP^\infty(M:\U \to \C)$ is $ (f,U): A \to B \in  \X \ / \sim \in  M: \U \to \C$, 
\item[Composition:] Given maps $[(f,U)]:A \to B$ and $[(g,V)]:B \to C$ in $\CP^{\infty}( M: \U \to \C)$, composition is defined as $$[(f,U)][(g,V)] := A \xrightarrow{f} U \oa B \xrightarrow{1 \oa g} U \oa (V \oa C) \xrightarrow{a_\oa} (U \oa V) \oa C \in \C$$
Graphically, this is represented by the following map in $\C$:
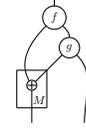
\begin{align*} 
 \begin{tikzpicture} 
	\begin{pgfonlayer}{nodelayer}
		\node [style=circle] (0) at (0.5, 2) {$f$};
		\node [style=circle] (1) at (1, 1) {$g$};
		\node [style=oa] (2) at (-0.25, -0.25) {};
		\node [style=none] (3) at (-0.75, 0.25) {};
		\node [style=none] (4) at (0.25, 0.25) {};
		\node [style=none] (5) at (0.25, -1) {};
		\node [style=none] (6) at (-0.75, -1) {};
		\node [style=none] (7) at (0, -0.75) {$M$};
		\node [style=none] (8) at (0.5, 2.75) {};
		\node [style=none] (9) at (1.5, -1.5) {};
		\node [style=none] (10) at (-0.25, -1.5) {};
	\end{pgfonlayer}
	\begin{pgfonlayer}{edgelayer}
		\draw (3.center) to (6.center);
		\draw (6.center) to (5.center);
		\draw (5.center) to (4.center);
		\draw (4.center) to (3.center);
		\draw [in=90, out=-45, looseness=1.00] (0) to (1);
		\draw [in=45, out=-135, looseness=1.00] (1) to (2);
		\draw [in=120, out=-135, looseness=1.00] (0) to (2);
		\draw [in=90, out=-45, looseness=1.00] (1) to (9.center);
		\draw (2) to (10.center);
		\draw (0) to (8.center);
	\end{pgfonlayer}
\end{tikzpicture}
\end{align*}
\item[Identity:]  $1_A$ is defined as $[A \xrightarrow{(u_\oa^L)^{-1}} \bot \oa A \to^{(n_\bot^M)^{-1} \oa 1} M(\bot) \oa A]  \in \X$
\end{description}
\end{definition}

To prove that $\CP^{\infty}(M: \U \to \C)$ is a category, we observe the following result about  unitary objects: 

\begin{lemma}
\label{Lemma: rho-tensor-par}
Suppose $C$ and $D$ are unitary objects. Then, the following diagrams
 commute:
\[
\xymatrixcolsep{1in}
\xymatrix{
M(C) \ox M(D) \ar[d]_{m_\ox^M}  \ar[r]^{M(\varphi) \ox M(\varphi)} \ar@{}[dddr]|{\bf (a)} & M(C^\dagger) \ox M(D^\dagger) \ar[d]^{\rho \ox \rho} \\
M(C \ox D) \ar[d]_{M(\varphi)} & M(C)^\dagger \ox M(D)^\dagger \ar[d]^{\lambda_\ox} \\
M((C \ox D)^\dagger) \ar[d]_{\rho} & (M(C) \oa M(D))^\dagger \ar[d]^{\mx^\dagger}\\
(M(C) \ox M(D))^\dagger \ar[r]_{(m_\ox^M)^\dagger} & (M(C) \ox M(D))^\dagger
} ~~~~~~~ \xymatrix{
 M(C) \oa M(D) \ar[r]^{M(\varphi_C) \oa M(\varphi_D)} \ar[d]_{n_\oa^{-1}} \ar@{}[dddr]|{\bf (b)} & M(C^\dagger) \oa M(D^\dagger) \ar[d]^{\rho \oa \rho} \\
M(C \oa D) \ar[d]_{M(\varphi)} & M(C)^\dagger \oa M(D)^\dagger \ar[d]^{\lambda_\oa} \\
M((C \oa D)^\dagger) \ar[d]_{\rho} & (M(C) \ox M(D))^\dagger \ar[d]^{(\mx^{-1})^\dagger}  \\
(M(C \oa D))^\dagger & (M(C) \oa M(D))^\dagger \ar[l]^{n_\oa^\dagger}
}
\]
\end{lemma}
\begin{proof}
\begin{align*}

\end{align*}

The commuting diagram (b) is proved similarly.
\end{proof}

Diagrammatic representation of Lemma \ref{Lemma: rho-tensor-par} (b):
\[
%
\]

\begin{proposition}
$\CP^\infty(M: \U \to \C)$ is a category.
\end{proposition}
\begin{proof}~
\begin{itemize}
\item{Composition is well-defined:}
 That is to say, if $(f,U) \sim (f',U')$ then $(f,U)(g,V) \sim (f',U')(g,V)$.  First we observe that: 
It suffices to show that:
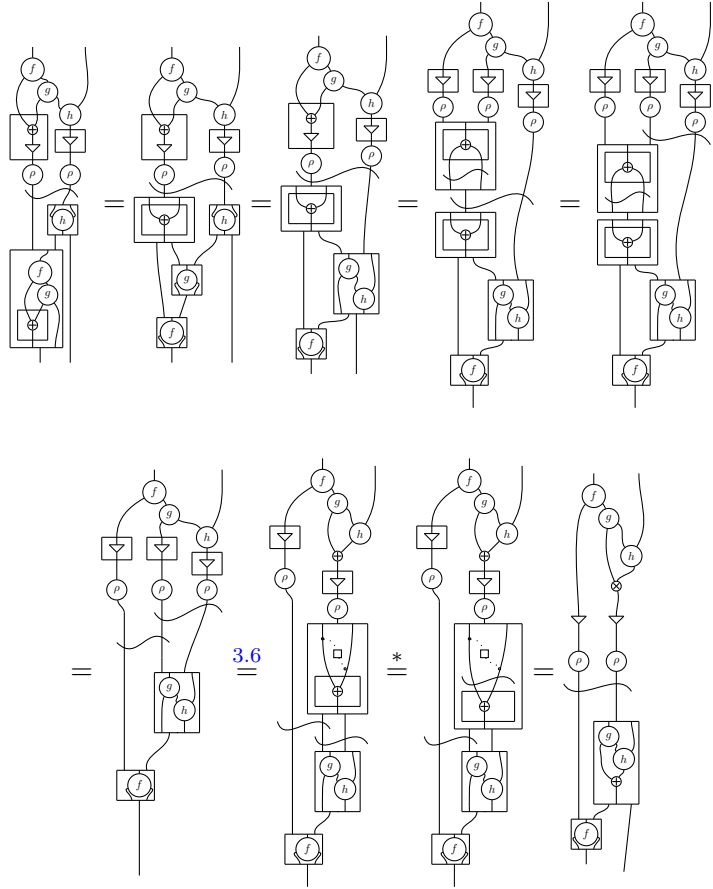
\begin{align*}
%
\end{align*}

$(g,V) \sim (g',V')  \Rightarrow (f,U)(g,V) \sim (f,U)(g',V')$ is proved similarly.

\item{Identity laws hold: $[(u_\oa^{L})^{-1} (n_\bot^M)^{-1}][f] = [f] = [f][(u_\oa^{L})^{-1} (n_\bot^M)^{-1}]$}
It suffices to prove the following: 
$$
%
$$

The proof uses the facts that $\rho$ is a monoidal transformation (diagram on the left), and that for any unitary object $U$, the diagram on the right holds:
\[ 
\xymatrixcolsep{4pc}
\xymatrix{
\top \ar[d]_{m_\top^M} \ar[rr]^{\lambda_\top} &  & \bot^\dagger \ar[d]_{(n_\bot^M)^\dagger}\ar@{=}[dr] \\
M(\top) \ar[r]_{M(\lambda_\top)} & M(\bot^\dagger) \ar[r]_{\rho} & M(\bot)^\dagger \ar[r]_{((n_\bot^M)^{-1})^\dagger} &  \bot^\dagger
} ~~~~~~~~~~ \xymatrix{
\top \ox U \ar[r]^{u_\ox^L} \ar[d]_{\m^{-1} \ox 1} & U \ar[d]^{(u_\oa^L)^{-1}} \\
\bot \ox U \ar[r]_{mx} & \bot \oa U
}
\]

\item{Composition is associative:}
Suppose $(f, U_1): A \to B,  (g, U_2): B \to C, (h, U_3): C \to D \in \CP^\infty(M: \U \to \C)$. Then,

$$
((f,U_1)(g, U_2)) (h, U_3) := \begin{tikzpicture} 
	\begin{pgfonlayer}{nodelayer}
		\node [style=circle] (0) at (-3.5, 1.75) {$f$};
		\node [style=circle] (1) at (-3, 0.75) {$g$};
		\node [style=circle] (2) at (-2.5, -0.25) {$h$};
		\node [style=none] (3) at (-2, -2) {};
		\node [style=oplus] (4) at (-4, -0.5) {};
		\node [style=oplus] (5) at (-3.5, -1.25) {};
		\node [style=none] (6) at (-3.5, -2) {};
		\node [style=none] (7) at (-3.5, 2.5) {};
		\node [style=none] (8) at (-4.5, -0) {};
		\node [style=none] (9) at (-3, -0) {};
		\node [style=none] (10) at (-3, -1.75) {};
		\node [style=none] (11) at (-4.5, -1.75) {};
		\node [style=none] (12) at (-4.25, -1.5) {$M$};
	\end{pgfonlayer}
	\begin{pgfonlayer}{edgelayer}
		\draw [style=none, in=90, out=-45, looseness=1.00] (0) to (1);
		\draw [style=none, in=90, out=-15, looseness=1.00] (1) to (2);
		\draw [style=none, in=90, out=-45, looseness=1.00] (2) to (3.center);
		\draw [style=none, in=45, out=-150, looseness=1.00] (1) to (4);
		\draw [style=none, in=-120, out=120, looseness=1.25] (4) to (0);
		\draw [style=none, in=150, out=-75, looseness=1.00] (4) to (5);
		\draw [style=none, in=75, out=-135, looseness=1.00] (2) to (5);
		\draw [style=none] (5) to (6.center);
		\draw [style=none] (7.center) to (0);
		\draw (9.center) to (10.center);
		\draw (10.center) to (11.center);
		\draw (11.center) to (8.center);
		\draw (8.center) to (9.center);
	\end{pgfonlayer}
\end{tikzpicture} \sim
\begin{tikzpicture}
	\begin{pgfonlayer}{nodelayer}
		\node [style=circle] (0) at (1.25, 0.25) {$f$};
		\node [style=oplus] (1) at (0, -0.5) {};
		\node [style=circle] (2) at (0.25, 1.75) {$h$};
		\node [style=none] (3) at (0.25, 2.5) {};
		\node [style=none] (4) at (-0.5, -2) {};
		\node [style=none] (5) at (1.75, -2) {};
		\node [style=circle] (6) at (0.75, 1) {$g$};
		\node [style=oplus] (7) at (-0.5, -1.25) {};
		\node [style=none] (8) at (-1, -0) {};
		\node [style=none] (9) at (0.5, -0) {};
		\node [style=none] (10) at (0.5, -1.75) {};
		\node [style=none] (11) at (-1, -1.75) {};
		\node [style=none] (12) at (0.25, -1.5) {$M$};
	\end{pgfonlayer}
	\begin{pgfonlayer}{edgelayer}
		\draw [style=none, in=90, out=-45, looseness=1.00] (2) to (6);
		\draw [style=none, in=90, out=-15, looseness=1.00] (6) to (0);
		\draw [style=none, in=90, out=-45, looseness=1.00] (0) to (5.center);
		\draw [style=none, in=90, out=-150, looseness=1.00] (6) to (1);
		\draw [style=none] (3.center) to (2);
		\draw [style=none, in=-150, out=15, looseness=1.25] (1) to (0);
		\draw [style=none, in=30, out=-124, looseness=1.25] (1) to (7);
		\draw [style=none, in=111, out=-135, looseness=1.00] (2) to (7);
		\draw [style=none] (7) to (4.center);
		\draw (8.center) to (9.center);
		\draw (9.center) to (10.center);
		\draw (10.center) to (11.center);
		\draw (11.center) to (8.center);
	\end{pgfonlayer}
\end{tikzpicture}
 =: (f, U_1) ((g, U_2)(h, U_3))
$$

Since $(U_1 \oa U_2) \oa U_3 \to^{a_\oa} U_1 \oa (U_2 \oa U_3)$ is a unitary isomorphism, by Lemma \ref{Lemma: unitary equivalence},  $(fg)h \sim f(gh) \in \C \Rightarrow ((f, U_1)(g, U_2))(h, U_3) = (f, U_1)((g, U_2)(h, U_3)) \in \CP^\infty(M:\U \to \C)$.
\end{itemize}
\end{proof}

There is a canonical functor $Q: \C \to  \CP^\infty(M: \U \to \C)$ of the original category into the category of channels:   

\begin{lemma}
\label{Lemma: embedding}
Let $M: \U \to \C $ be a mixed unitary category, then there is a canonical functor: 
\[ Q: \C \to  \CP^\infty(M: \U \to \C) ; 
\begin{matrix}
\xymatrix{
A \ar[d]_{f} \\
B
} \end{matrix}   \mapsto  \begin{matrix}
\xymatrix{
A \ar[d]^{f (u_\oa^L)^{-1}((n_\bot^M)^{-1} \oa 1)} \\
M(\bot) \oa B
}
\end{matrix}
\]
\end{lemma}
\begin{proof}
$Q$ is a preserves identity maps and composition because $ f(u_\oa^L)^{-1} \sim f \sim (u_\oa^L)^{-1} f $. 
\end{proof}
There is no reason why this functor should be faithful and, indeed, in many cases it will {\em not\/} be faithful \cite{CH16}.

\subsection{Structures inherited by the $\CP$-infinity construction}
\label{Sec: CP-inf structures}

While the $\CP^\infty$ construction described in the previous section produces a category of equivalence classes of Kraus morphisms in a MUC, it is desirable to know if the resulting category inherits dagger and LDC structures from the base category.

The following table summarizes the structures inherited by $\CP^\infty(M: \U \to \C)$ from $ M: \U \to \C$ for different variants of MUCs:
\begin{center}
\begin{tabular}{ |c|c| }
\hline
$ M: \U \to \C $ & $ \CP^\infty$ ($M: \U \to \C$)   \\ 
\hline \hline
mixed unitary category & isomix category  \\  \hline
 $*$-mixed unitary category with unitary duals &  $*$-mixed unitary category with unitary duals\\ \hline
 $\dagger$-symmetric monoidal category & symmetric monoidal category    \\ \hline
 $\dagger$-compact closed category & $\dagger$-compact closed category ($\CP^\infty(\X) \simeq$ CPM$(\X)$)\\ \hline
\end{tabular}
\end{center}    

Theorem \ref{thm: MUC is isomix} shows that the CP-infinity construction on any mixed unitary category produces an isomix category. Lemma \ref{thm: MUC is dagger isomix} shows that if the base category has unitary duals, then the resulting category is dagger-isomix. Finally, it is shown in Lemma \ref{thm: CP is functorial} that the CP-infinity construction is functorial in the presence of unitary duals. 

\begin{theorem}
\label{thm: MUC is isomix}
Given any mixed unitary categoy $M: \U \to \C$, $\CP^\infty(M: \U \to \C)$ is an isomix category.
\end{theorem}
\begin{proof}
We know that $\CP^\infty(M: \U \to \C)$ is well-defined category. Indeed it has two tensors: $\widehat{\ox}$ and $\widehat{\oa}$ given by the following Kraus maps:
$$f \widehat{\ox} g := \begin{tikzpicture}
	\begin{pgfonlayer}{nodelayer}
		\node [style=circle, scale=2] (0) at (-2, 1) {};
		\node [style=none] (12) at (-2, 1) {$f$};
		\node [style=circle, scale=2] (1) at (0, 1) {};
		\node [style=none] (13) at (0, 1) {$g$};
		\node [style=ox] (2) at (-1, 2) {};
		\node [style=oa] (3) at (-2, -0.5) {};
		\node [style=ox] (4) at (0, -0.5) {};
		\node [style=none] (5) at (0, -1.5) {};
		\node [style=none] (6) at (-2, -1.5) {};
		\node [style=none] (7) at (-1, 3) {};
		\node [style=none] (8) at (-2.5, -0) {};
		\node [style=none] (9) at (-1.5, -0) {};
		\node [style=none] (10) at (-1.5, -1) {};
		\node [style=none] (11) at (-2.5, -1) {};
	\end{pgfonlayer}
	\begin{pgfonlayer}{edgelayer}
		\draw (7.center) to (2);
		\draw [in=90, out=-90, looseness=1.00] (1) to (3);
		\draw [in=105, out=-75, looseness=1.00] (0) to (4);
		\draw [in=90, out=-165, looseness=1.00] (2) to (0);
		\draw [in=90, out=-15, looseness=1.00] (2) to (1);
		\draw (4) to (5.center);
		\draw (3) to (6.center);
		\draw [bend right, looseness=1.25] (0) to (3);
		\draw [bend left, looseness=1.25] (1) to (4);
		\draw (8.center) to (9.center);
		\draw (9.center) to (10.center);
		\draw (10.center) to (11.center);
		\draw (11.center) to (8.center);
	\end{pgfonlayer}
\end{tikzpicture}
\hspace*{1cm}
f \widehat{\oa} g := \begin{tikzpicture}
	\begin{pgfonlayer}{nodelayer}
		\node [style=circle, scale=2] (0) at (-2, 1) {};
		\node [style=none] (12) at (-2, 1) {$f$};
		\node [style=circle, scale=2] (1) at (0, 1) {};
		\node [style=none] (13) at (0, 1) {$g$};
		\node [style=oa] (2) at (-1, 2) {};
		\node [style=oa] (3) at (-2, -0.5) {};
		\node [style=oa] (4) at (0, -0.5) {};
		\node [style=none] (5) at (0, -1.5) {};
		\node [style=none] (6) at (-2, -1.5) {};
		\node [style=none] (7) at (-1, 3) {};
		\node [style=none] (8) at (-2.5, -0) {};
		\node [style=none] (9) at (-1.5, -0) {};
		\node [style=none] (10) at (-1.5, -1) {};
		\node [style=none] (11) at (-2.5, -1) {};
	\end{pgfonlayer}
	\begin{pgfonlayer}{edgelayer}
		\draw (7.center) to (2);
		\draw [in=90, out=-90, looseness=1.00] (1) to (3);
		\draw [in=105, out=-75, looseness=1.00] (0) to (4);
		\draw [in=90, out=-165, looseness=1.00] (2) to (0);
		\draw [in=90, out=-15, looseness=1.00] (2) to (1);
		\draw (4) to (5.center);
		\draw (3) to (6.center);
		\draw [bend right, looseness=1.25] (0) to (3);
		\draw [bend left, looseness=1.25] (1) to (4);
		\draw (8.center) to (9.center);
		\draw (9.center) to (10.center);
		\draw (10.center) to (11.center);
		\draw (11.center) to (8.center);
	\end{pgfonlayer}
\end{tikzpicture}$$

The units for $\widehat{\ox}$ and $\widehat{\oa}$ are $\top$ and $\bot$ respectively. 

The linear distribution maps and all the basic basic natural isomorphisms  are inherited from $\X$ by composing each one of them with $(u_\oa^L)^{-1}$ i.e.,

\[
\hfil \infer{ A \widehat{\ox} (B \widehat{\ox} C)  \xrightarrow{a_{\widehat{\ox}} := a_\ox (u_\oa^L)^{-1}} (A \widehat{\ox} B) \widehat{\ox} C \in  \CP^\infty(M: \U \to \C)} { A \ox (B \ox C) \xrightarrow{a_\ox} (A \ox B) \ox C \xrightarrow{(u_\oa^L)^{-1}}  \C \ / \sim}
\]

We prove that the associators and the other maps as defined above are natural isomorphisms in $\CP^\infty(M: \U \to \C)$:  From Lemma $\ref{Lemma: embedding}$, $Q: \C \hookrightarrow \CP^\infty(M: \U \to \C)$ is functorial which means that all commuting diagrams and isomorphisms are preserved. It remains to show that Q preserves the linear structure and the mix map:

\begin{itemize}

\item $Q$ preserves $\ox$: Suppose $f:A \to A'$ and $g:B \to B' \in \C$. Then,  $Q(f) \widehat{\ox} Q(g) = Q(f \ox g)$: 

\begin{align*}
Q(f) \widehat{\ox} Q(g) &:= A \widehat{\ox} B 
\xrightarrow{f (u_\oa^L)^{-1} \widehat{\ox} g (u_\oa^L)^{-1}}
(\bot \widehat{\oa} \bot) \widehat{\oa} (A' \widehat{\ox}  B') \\
Q(f \ox g) &:=  A \widehat{\ox} B \xrightarrow{ (f \ox g) u_\oa^{-1}} \bot \widehat{\oa} (A' \widehat{\ox} B') 
\end{align*}

Since, $\bot \oa \bot \xrightarrow{u_\oa^L} \bot$ is a unitary isomorphism and, $f (u_\oa^L)^{-1} \widehat{\ox} g (u_\oa^L)^{-1} (u_\oa^L \oa 1) = (f \ox g) (u_\oa^L)^{-1} \in \C$, by Lemma \ref{Lemma: unitary equivalence}, 

\[ f (u_\oa^L)^{-1}) \widehat{\ox} (g (u_\oa^L)^{-1} \sim (f \ox g) (u_\oa^L)^{-1} \]

Therefore, $Q(f) \widehat{\ox} Q(g) = Q(f \ox g)$. Similarly,  $Q(f) \widehat{\oa} Q(g) = Q(f \oa g)$.

\item $Q$ preserves all basic natural isomorphisms  (associators, unitors, symmetry maps, mix map) and linear distributions:

To prove that $a_{\widehat{\ox}}$ is natural in $\CP^\infty(M:\U \to \C)$, we need to prove that the following diagram commutes in $\CP^\infty( M: \U \to \C )$:

\[
\xymatrix{
A \widehat{\ox} (B \widehat{\ox} C) \ar[r]^{a_{\widehat{\ox}}}  \ar[d]_{(f \widehat{\ox} g) \widehat{\ox} h}
& (A \widehat{\ox} B) \widehat{\ox} C \ar[d]^{f \widehat{\ox} (g \widehat{\ox} h)} \\
A' \widehat{\ox} (B' \widehat{\ox} C') \ar[r]^{a_{\widehat{\ox}}}
& (A' \widehat{\ox} B') \widehat{\ox} C'
}
\]
In other words, we need to show that the two compositions in $\C$ are equivalent as Kraus maps. This is follows from Lemma \ref{Lemma: unitary equivalence} as there is a unitary isomorphism between the ancillary objects $\bot \oa (U_1 \oa (U_2 \oa U_3))$ and $\bot \oa (U_1 \oa U_2) \oa U_3$.  Similarly, we can show that the other basic linearly distrbutive transformations as defined are natural transformations. Since $Q$ is functorial, it preserves isomorphisms and commuting diagrams so that the coherence diagrams automatically commute.
\end{itemize}
\end{proof}

Observe that our $\CP^\infty$-construction on $M: \U \to \C$  coincides with the original $\CP^\infty$-construction \cite{CH16} when $M: \U \to \C$ is dagger monoidal category $\U$ is $\dagger$-monoidal and $M = id$. 

$\CP^\infty(M: \U \to \C)$, in general, does not have dagger even when $\C$ is a $\dagger$-isomix category. However, if $M: \U \to \C$ is a $*$-MUdC that is a mixed unitary category in which every object in $\U$ has unitary duals and $\C$ is a $\dagger$-isomix $*$-autonomous category, then $\CP^\infty(M: \U \to \C)$ has an obvious dagger as shown in the following theorem:

\begin{lemma}
\label{thm: MUC is dagger isomix}

If {\em M:} $\U \to \C$ is a $*${\em -MUdC} then $\CP^\infty($M: $\U \to \C)$ is a $\dagger$-isomix category and the following is a $*${\em -MUdC}.
\[ N: \U \to \CP^\infty(M: \U \to \C) ; ~~~~~~

\] 
\end{lemma}
\begin{proof}
(Sketch) We first oberve that $\CP^\infty(M: \U \to \C)$ is a isomix $\dagger$-LDC. We have already proven that $\CP^\infty(M: \U \to \C)$ is an isomix category. Suppose  $\X$ is a $*$-MUdC, then the $\dagger$ functor for $\CP^\infty(M: \U \to \C)$ is defined as follows: Suppose $(f, U): A \to B$ and $(\eta, \epsilon): V \dashvv_{~u} U$ with $A' \dashvv A$ and $B' \dashvv B \in \C$ then,
\[
\dagger : \CP^\infty(M: \U \to \C)^{\rm op} \to \CP^\infty(M: \U \to \C); %
\]

We need to prove that $\dagger$ is well-defined:  $f \sim g \Rightarrow f^\dagger \sim g^\dagger$. 
\begin{align*}
%
\end{align*}

The equality is proved by using the snake diagrams and {\bf [U.2]}, {\bf [U.5](b)}, and {\bf [Udual.]}. 

Suppose $f: A \rightarrow U_1 \oplus B$ and $g: B \to U_2 \oplus C$ with $(\eta_1, \epsilon_1): U_1 \dashvv_{~u} V_1$ and $(\eta_2, \epsilon_2): U_2 \dashvv_{~u} V_2$, then $\dagger$ preserves composition, that is  $(fg)^\dagger = g^\dagger f^\dagger$:

\[
\hfil \infer{ (fg)^\dagger: C^\dagger \rightarrow (V_1 \otimes V_2)^\dagger \oplus A^\dag} { (fg): A \to (U_1 \oplus U_2) \oplus C}
\]

\[
\hfil \infer{ (g^\dagger f^\dagger): C^\dagger \rightarrow (V_2^\dag \otimes V_1^\dagger) \oplus A^\dag } { g^\dagger: C^\dagger \to U_2^\dagger \oplus B^\dagger ~~~~~~~~~ f^\dag: B^\dag \to U_1^\dag \oplus A^\dag}
\]

To prove that $(fg)^\dagger = (g^\dagger f^\dagger)$ in $\CP^\infty(M: \U \to \C)$, represent the maps in circuit calculus and fuse the $\dagger$-boxes. Once the $\dagger$-boxes are fused, use 
Lemma \ref{Lemma: unitary equivalence} to show that both Kraus operations belong to the same equivalence. $\dag$ preserves identity map since $((u_\oplus^R)^{-1}, u_\oplus^L): \top \dashvv_{~u} \bot$. Hence, $\dag$ is a functor.

All the basic natural isomorphisms associated with $\dagger$ functor - $\lambda_\oplus, \lambda_\otimes, \lambda_\bot, \lambda_\top, \iota$ - are lifted from $\C$ using $Q: \C \hookrightarrow \CP^\infty(M: \U \to \C)$ which is defined in Lemma \ref{Lemma: embedding}. The lifted morphisms are natural in $\CP^\infty(M : \U \to \C)$ since their ancillaries are unitarily isomorphic. Since $\dagger$ is functorial, all commuting diagrams are preserved. By the same argument, unitary structure is preserved under $Q$.

Thus, $\CP^\infty(M: \U \to \C)$ is a mixed unitary category: as $Q$ preserves all unitary linear duals this makes  $\CP^\infty(M: \U \to \C)$ a $*$-MUdC. 
\end{proof}

\begin{lemma}
\label{thm: CP is functorial}

The $\CP^\infty$-construction is functorial on MUCs with unitary duals, that is, $*$-MUdCs.
\end{lemma}
\begin{proof}
Let $M: \U \to \C$ and $M': \U' \to \C'$ be $*$-MUdCs and the following square be a MUC morphism:
\[
\xymatrix{
\U  \ar[rr]^{M} \ar[d]_{F_u} & \ar@{}[d]|{\Downarrow~\alpha} & \C \ar[d]^{F} \\
\U' \ar[rr]_{M'} & &\C'
}
\]
$F_u$ and $F$ are $\dagger$-isomix functors and $F_u$ preserves unitary structure i.e., $F_u(\varphi_A) \rho^{F_u} = \varphi_{F_u(A)}$ and ($n_\bot^{F_u}$ or $m_\top^{F_u}$) is a unitary isomorphism. $\alpha$ is a $\dagger$-linear natural isomorphism. 

Then, the $\CP^\infty$-construction is functorial if there is a MUC morphism:
\[
\xymatrix{
\U  \ar[rr]^{MQ} \ar[d]_{F_u} & \ar@{}[d]|{\Downarrow~\alpha':=?} & \CP^\infty(M: \U \to \C) \ar[d]^{G:=?} \\
\U' \ar[rr]_{M'Q} & & \CP^\infty(M': \U' \to \C')
}
\]

Recall the functor $Q: \C \hookrightarrow \CP^\infty(M: \U \to \C)$ from Lemma \ref{Lemma: embedding}. Define $G: \CP^\infty(M: \U \to \C) \to \CP^\infty(M': \U' \to \C')$ as follows:

\[ G: \CP^\infty(M: \U \to \C) \to \CP^\infty(M': \U' \to \C') ; ~~~~
\begin{matrix}
\xymatrix{
A \ar[d]_{[(f, U)]} \\
B
}
\end{matrix}
\mapsto
\begin{matrix}
\xymatrix{
F(A) \ar[d]^{[(F(f) n_\oa^F (\alpha \oa 1), F_u(U))]} \\
F(B)
}
\end{matrix}
\]
The action of functor $G: \CP^\infty(M: \U \to \C) \to \CP^\infty(M': \U' \to \C')$ on maps is drawn as follows:
\[
\begin{matrix} 
\begin{tikzpicture}
	\begin{pgfonlayer}{nodelayer}
		\node [style=circle] (0) at (0, 1) {$f$};
		\node [style=none] (1) at (-0.5, -0) {};
		\node [style=none] (2) at (0.5, -0) {};
		\node [style=none] (3) at (0, 2) {};
	\end{pgfonlayer}
	\begin{pgfonlayer}{edgelayer}
		\draw (3.center) to (0);
		\draw [bend right, looseness=1.00] (0) to (1.center);
		\draw [bend left, looseness=1.00] (0) to (2.center);
	\end{pgfonlayer}
\end{tikzpicture}
\end{matrix}
\mapsto
\begin{matrix}
\begin{tikzpicture} [scale=1.5] 
	\begin{pgfonlayer}{nodelayer}
		\node [style=circle] (0) at (0, 1.25) {$f$};
		\node [style=none] (1) at (0.75, -1) {};
		\node [style=none] (2) at (0, 2.25) {};
		\node [style=circle] (3) at (-0.75, -0) {$\alpha$};
		\node [style=none] (4) at (-0.75, -1) {};
		\node [style=none] (5) at (-1.25, 0.5) {};
		\node [style=none] (6) at (1, 0.5) {};
		\node [style=none] (7) at (1, 1.75) {};
		\node [style=none] (8) at (-1.25, 1.75) {};
		\node [style=none] (9) at (-1, 1.5) {$F$};
		\node [style=none] (10) at (0.75, 2.25) {F(A)};
		\node [style=none] (11) at (1.5, -0.75) {$F(B)$};
		\node [style=none] (12) at (-1.75, -0.75) {$M'(F_u(U))$};
	\end{pgfonlayer}
	\begin{pgfonlayer}{edgelayer}
		\draw (2.center) to (0);
		\draw [in=90, out=-15, looseness=1.00] (0) to (1.center);
		\draw [in=90, out=-165, looseness=1.25] (0) to (3);
		\draw (3) to (4.center);
		\draw (8.center) to (7.center);
		\draw (6.center) to (7.center);
		\draw (6.center) to (5.center);
		\draw (5.center) to (8.center);
	\end{pgfonlayer}
\end{tikzpicture}
\end{matrix}
\] 
The natural isomorphism $\alpha$ lifts to $\CP^{\infty}(M': \U' \to \C')$ as follows: \[ \alpha':= [(\alpha_u (u_\oa^L)^{-1} ((n_\bot^{M'})^{-1} \oa 1), \bot)]: F(M(U)) \to M'(F_u(U)) \]

It is immediate that $\alpha'_U:  (G(Q(M(U_1) := F(M(U_1))) \to (Q(M'(F_u(U_2))) := M'(F_u(U_1))$ is an isomorphism. Let $U_1 \to^{f} U_2 \in \U$. To prove that $\alpha'$ is natural, we show that the following diagram commutes in $\CP^\infty(M': \U' \to \C')$.
\[
\xymatrixcolsep{15pc}
\xymatrix{
F(M(U_1)) \ar[d]_{\alpha'_{U_1}} \ar[r]^{[( (F(M(f)(u_\oa^L)^{-1}((n_\bot^M)^{-1} \oa 1)) n_\oa^F (\alpha \oa 1)),F_u(\bot))]} &  F(M(U_2)) \ar[d]^{\alpha'_{U_2}}\\
M'(F_u(U_1)) \ar[r]_{[(M'(F_u(f)) (u_\oa^L)^{-1} ((n_\bot^{M'})^{-1} \oa 1), \bot)]} & M'(F_u(U_2))
}
\]

The underlying Kraus maps for both the compositions are equivalent since $F_u(\bot) \to^{(n_\bot^{F_u})^{-1}} \bot$ is unitarily isomorphic: 

\[
\begin{tikzpicture} 
	\begin{pgfonlayer}{nodelayer}
		\node [style=circle] (0) at (0, 2) {$\alpha$};
		\node [style=circle] (1) at (0, 0.25) {$f$};
		\node [style=none] (2) at (1, -0) {$F_uM'$};
		\node [style=none] (3) at (-0.5, 0.75) {};
		\node [style=none] (4) at (-0.5, -0.25) {};
		\node [style=none] (5) at (1.5, -0.25) {};
		\node [style=none] (6) at (1.5, 0.75) {};
		\node [style=circle] (7) at (-0.75, -1.5) {$\bot$};
		\node [style=none] (8) at (0, -1.75) {$M'$};
		\node [style=none] (9) at (-1.25, -1) {};
		\node [style=none] (10) at (-1.25, -2) {};
		\node [style=none] (11) at (0.25, -2) {};
		\node [style=none] (12) at (0.25, -1) {};
		\node [style=none] (13) at (1.25, -4.25) {};
		\node [style=circle, scale=0.2] (14) at (0.5, -0.75) {};
		\node [style=circle] (15) at (-2.25, 0.25) {$\bot$};
		\node [style=circle, scale=0.2] (16) at (0, 1.25) {};
		\node [style=none] (17) at (-2.75, 0.75) {};
		\node [style=none] (18) at (-1.25, 0.75) {};
		\node [style=none] (19) at (-1.25, -0.25) {};
		\node [style=none] (20) at (-2.75, -0.25) {};
		\node [style=none] (21) at (-1.5, -0) {$M'$};
		\node [style=oplus] (22) at (-1.75, -3.25) {};
		\node [style=none] (23) at (-1.75, -4.25) {};
		\node [style=none] (24) at (-2.5, -2.75) {};
		\node [style=none] (25) at (-1, -2.75) {};
		\node [style=none] (26) at (-1, -3.75) {};
		\node [style=none] (27) at (-2.5, -3.75) {};
		\node [style=none] (28) at (-2.25, -3.5) {$M$};
		\node [style=none] (29) at (-3, -4.25) {$M(\bot \oa \bot)$};
		\node [style=none] (30) at (2.25, -4) {$M'(F_u(U_2)$};
		\node [style=none] (31) at (0, 2.5) {};
		\node [style=none] (32) at (0.5, 2.5) {$F(M(U_1)$};
	\end{pgfonlayer}
	\begin{pgfonlayer}{edgelayer}
		\draw (3.center) to (6.center);
		\draw (6.center) to (5.center);
		\draw (5.center) to (4.center);
		\draw (4.center) to (3.center);
		\draw (9.center) to (12.center);
		\draw (12.center) to (11.center);
		\draw (11.center) to (10.center);
		\draw (10.center) to (9.center);
		\draw [in=90, out=-60, looseness=1.25] (1) to (13.center);
		\draw (1) to (0);
		\draw [dotted, bend right=45, looseness=1.00] (14) to (7);
		\draw [dotted, in=75, out=174, looseness=1.00] (16) to (15);
		\draw (17.center) to (18.center);
		\draw (18.center) to (19.center);
		\draw (19.center) to (20.center);
		\draw (20.center) to (17.center);
		\draw [in=120, out=-90, looseness=0.75] (15) to (22);
		\draw [in=45, out=-90, looseness=1.00] (7) to (22);
		\draw (22) to (23.center);
		\draw (24.center) to (25.center);
		\draw (25.center) to (26.center);
		\draw (26.center) to (27.center);
		\draw (27.center) to (24.center);
		\draw (31.center) to (0);
	\end{pgfonlayer}
\end{tikzpicture}  \sim \begin{tikzpicture} 
	\begin{pgfonlayer}{nodelayer}
		\node [style=circle] (0) at (1.75, 1.5) {$f$};
		\node [style=circle] (1) at (-0.25, -0) {$\bot$};
		\node [style=circle] (2) at (-0.25, -1) {$\alpha$};
		\node [style=circle] (3) at (1.75, -1) {$\alpha$};
		\node [style=circle] (4) at (1, -2.75) {$\bot$};
		\node [style=none] (5) at (1.75, 2.5) {};
		\node [style=none] (6) at (1.75, -4.75) {};
		\node [style=none] (7) at (1, -3.75) {};
		\node [style=none] (8) at (-0.25, -3.75) {};
		\node [style=circle, scale=0.2] (9) at (1.75, 0.75) {};
		\node [style=circle, scale=0.2] (10) at (1.75, -1.75) {};
		\node [style=none] (11) at (0, -2.25) {};
		\node [style=none] (12) at (0, -3.25) {};
		\node [style=none] (13) at (1.5, -3.25) {};
		\node [style=none] (14) at (1.5, -2.25) {};
		\node [style=none] (15) at (0.5, -3) {$M$};
		\node [style=none] (16) at (-0.75, 2) {};
		\node [style=none] (17) at (-0.75, -0.5) {};
		\node [style=none] (18) at (2.75, -0.5) {};
		\node [style=none] (19) at (2.75, 2) {};
		\node [style=none] (20) at (-0.25, 1.5) {$MF$};
		\node [style=oplus] (21) at (0.5, -4.25) {};
		\node [style=none] (22) at (0.5, -4.75) {};
		\node [style=none] (23) at (-1.25, -4.75) {$M'(F_u(\bot) \oa \bot)$};
		\node [style=none] (24) at (2.5, -4.75) {$M'(F_u(U_2))$};
		\node [style=none] (25) at (0.75, 2.5) {$F(M(U_1))$};
	\end{pgfonlayer}
	\begin{pgfonlayer}{edgelayer}
		\draw (5.center) to (0);
		\draw (0) to (3);
		\draw (3) to (6.center);
		\draw [dotted, in=75, out=172, looseness=0.75] (9) to (1);
		\draw (1) to (2);
		\draw [dotted, bend right, looseness=1.00] (10) to (4);
		\draw (4) to (7.center);
		\draw (2) to (8.center);
		\draw (11.center) to (12.center);
		\draw (12.center) to (13.center);
		\draw (13.center) to (14.center);
		\draw (14.center) to (11.center);
		\draw (16.center) to (19.center);
		\draw (19.center) to (18.center);
		\draw (18.center) to (17.center);
		\draw (17.center) to (16.center);
		\draw [in=165, out=-90, looseness=1.00] (8.center) to (21);
		\draw [in=30, out=-90, looseness=1.00] (7.center) to (21);
		\draw (21) to (22.center);
	\end{pgfonlayer}
\end{tikzpicture}
\]
Hence, the diagram commutes and $\alpha'$ is natural in $\CP^\infty(M': \U' \to \C')$.
\end{proof}


\section{Environment structure for mixed unitary categories}
\label{Sec: Env maps}

In this section, we describe when a given isomix category is of form $\CP^\infty( M:\U \to \C)$ for a MUC $M: \U \to \C$ by generalizing the notion of environment structures from $\dagger$-symmetric monoidal categories (see \cite{CH16}) to mixed unitary categories. We then show that an environment structure over $M$ which has purification is isomorphic to $\CP^\infty(M: \U \to \C)$. 

An environment structure for a $\dagger$-symmetric monoidal category equips the category with the ability to discard any object in the category in a natural way, that is, there exists a natural transformation $\envmap_A: A \to I$ in the category. By generalizing the notion of environment structures to MUCs, one notices that it suffices to require discard maps only for the unitary objects $(A \simeq A^\dagger)$ and not for every object. 

\subsection{Environment structure and examples}
The first task is to define environment structures and show that the $\CP$-infinity construction gives an example of these structures. 

\begin{definition} 
\label{Definition: env structure}
An {\bf environment structure} for a mixed unitary category $M: \U \to \C$ is a strict isomix functor $F: \C \to \D$ where $\D$ is an isomix category and 
 a family of discard maps $\envmap_U: M(U) \to \bot$ indexed by objects $ U \in \U$ such that the following conditions hold:
\begin{enumerate}[{\bf [Env.1]}]
\item For $U, V \in \U$, the following diagrams commute, that is, $\envmap$ is closed to $\ox$ and $\oa$.

\[ \mbox{\bf (a)} ~~~~~~ \xymatrixcolsep{3.5pc}
\xymatrix{
MF(U) \ox MF(V) \ar[r]^{\mx} \ar[d]_{m_\ox} & M(U) \oa M(V) \ar[r]^{\envmap \oa \envmap} & \bot \oa \bot \ar[d]^{u_\oa} \\
MF(U \ox V) \ar[rr]_{\envmap} & & \bot
}
\]
\[ \mbox{\bf (b)} ~~~~~~~~~~~~~~~~~~~~~~ \xymatrixcolsep{3pc} \xymatrix{
F(M(U \oa V)) \ar@/^1pc/[drr]^{\envmap } \ar[d]_{n_\oa}\\
F(M(U)) \oa F(M(V)) \ar[r]_{\envmap \oa \envmap} & \bot \oa \bot \ar[r]_{u_\oa} & \bot
} \]
\item  Kraus maps $(f, U) \sim (g, V) \in \C$ if and only if  the following diagram commutes:
\[ \xymatrix{
& F(M(A)) \ar[ld]_{F(M(f))} \ar[dr]^{F(M(g))} & \\
F(M(U \oa B)) \ar[d]_{\nu_\ox} &  & F(M(U \oa B)) \ar[d]_{\nu_\ox} \\
F(M(U)) \oa MF(B) \ar[d]_{\envmap \oa 1} &  & F(M(U)) \oa F(M(B)) \ar[d]_{\envmap \oa 1} \\
\bot \oa F(M(B)) \ar@{=}[rr] &  & \bot \oa F(M(B)) 
} \]
\end{enumerate}
\end{definition}

The conditions are represented diagrammatically as follows:
\[
\mbox{ \bf{ [Env.1a]}}~~~~~ 
\begin{tikzpicture}
	\begin{pgfonlayer}{nodelayer}
		\node [style=ox] (0) at (1.75, 1.5) {};
		\node [style=none] (1) at (1.75, -0) {};
		\node [style=none] (2) at (1, 3) {};
		\node [style=none] (3) at (2.5, 3) {};
		\node [style=none] (4) at (0.5, 2.25) {};
		\node [style=none] (5) at (0.5, 0.75) {};
		\node [style=none] (6) at (2.75, 0.75) {};
		\node [style=none] (7) at (2.75, 2.25) {};
		\node [style=none] (8) at (2.25, 1) {$MF$};
		\node [style=none] (9) at (1.5, -0) {};
		\node [style=none] (10) at (2, -0) {};
		\node [style=none] (11) at (1.6, -0.1) {};
		\node [style=none] (12) at (1.9, -0.1) {};
		\node [style=none] (13) at (1.7, -0.2) {};
		\node [style=none] (14) at (1.8, -0.2) {};
		\node [style=none] (15) at (1.75, -0.5) {};
		\node [style=none] (16) at (1.75, -1) {};
		\node [style=circle, scale=2] (17) at (1.75, -0.1) {};
	\end{pgfonlayer}
	\begin{pgfonlayer}{edgelayer}
		\draw (4.center) to (7.center);
		\draw (7.center) to (6.center);
		\draw (6.center) to (5.center);
		\draw (5.center) to (4.center);
		\draw (0) to (1.center);
		\draw [bend right=15, looseness=1.00] (2.center) to (0);
		\draw [bend right=15, looseness=1.00] (0) to (3.center);
		\draw (15.center) to (16.center);
		\draw (9.center) to (10.center);
		\draw (11.center) to (12.center);
		\draw (13.center) to (14.center);
	\end{pgfonlayer}
\end{tikzpicture} = \begin{tikzpicture}
	\begin{pgfonlayer}{nodelayer}
		\node [style=none] (0) at (-3, 3) {};
		\node [style=none] (1) at (-1, 3) {};
		\node [style=circle, scale=0.2] (2) at (-3, 2.5) {};
		\node [style=circle, scale=0.2] (3) at (-1, 1.5) {};
		\node [style=map] (4) at (-2, 2) {};
		\node [style=none] (5) at (-3, 1) {};
		\node [style=none] (6) at (-1, 1) {};
		\node [style=none] (7) at (-3.25, 1) {};
		\node [style=none] (8) at (-2.75, 1) {};
		\node [style=none] (9) at (-3.15, 0.9) {};
		\node [style=none] (10) at (-2.85, 0.9) {};
		\node [style=none] (11) at (-3.05, 0.8) {};
		\node [style=none] (12) at (-2.95, 0.8) {};
		\node [style=none] (13) at (-1.25, 1) {};
		\node [style=none] (14) at (-0.75, 1) {};
		\node [style=none] (15) at (-1.15, 0.9) {};
		\node [style=none] (16) at (-0.85, 0.9) {};
		\node [style=none] (17) at (-1.05, 0.8) {};
		\node [style=none] (18) at (-0.95, 0.8) {};
		\node [style=none] (19) at (-1, 0.5) {};
		\node [style=none] (20) at (-3, 0.5) {};
		\node [style=none] (21) at (-1, -0.25) {};
		\node [style=circle] (22) at (-3, -0.25) {$\bot$};
		\node [style=circle, scale=2] (23) at (-3, 0.85) {};
		\node [style=circle, scale=2] (24) at (-1, 0.85) {};
	\end{pgfonlayer}
	\begin{pgfonlayer}{edgelayer}
		\draw (0.center) to (5.center);
		\draw (1.center) to (6.center);
		\draw [dotted, bend left=45, looseness=1.25] (2) to (4);
		\draw [dotted, bend right=45, looseness=1.25] (4) to (3);
		\draw (7.center) to (8.center);
		\draw (9.center) to (10.center);
		\draw (11.center) to (12.center);
		\draw (13.center) to (14.center);
		\draw (15.center) to (16.center);
		\draw (17.center) to (18.center);
		\draw (20.center) to (22);
		\draw (19.center) to (21.center);
	\end{pgfonlayer}
\end{tikzpicture}  ~~~~~~~~~~ 
\mbox{ \bf{ [Env.1b]}}~~~~~\begin{tikzpicture}
	\begin{pgfonlayer}{nodelayer}
		\node [style=oa] (0) at (1.75, 1.5) {};
		\node [style=none] (1) at (1.75, 3) {};
		\node [style=none] (2) at (1, 0) {};
		\node [style=none] (3) at (2.5, 0) {};
		\node [style=none] (4) at (0.5, 0.75) {};
		\node [style=none] (5) at (0.5, 2.25) {};
		\node [style=none] (6) at (2.75, 2.25) {};
		\node [style=none] (7) at (2.75, 0.75) {};
		\node [style=none] (8) at (0.85, 1) {$MF$};
		\node [style=none] (9) at (2.25, -0) {};
		\node [style=none] (10) at (2.75, -0) {};
		\node [style=none] (11) at (2.35, -0.1) {};
		\node [style=none] (12) at (2.65, -0.1) {};
		\node [style=none] (13) at (2.45, -0.2) {};
		\node [style=none] (14) at (2.55, -0.2) {};
		\node [style=none] (15) at (2.5, -0.5) {};
		\node [style=none] (16) at (2.5, -1) {};
		\node [style=circle, scale=2] (17) at (2.5, -0.1) {};
		\node [style=none] (18) at (0.75, -0) {};
		\node [style=none] (19) at (1.25, -0) {};
		\node [style=none] (20) at (0.85, -0.1) {};
		\node [style=none] (21) at (1.15, -0.1) {};
		\node [style=none] (22) at (0.95, -0.2) {};
		\node [style=none] (23) at (1.05, -0.2) {};
		\node [style=none] (24) at (1, -0.5) {};
		\node [style=circle, scale=2] (25) at (1, -0.1) {};
		\node [style=circle] (26) at (1, -1) {$\bot$};
	\end{pgfonlayer}
	\begin{pgfonlayer}{edgelayer}
		\draw (4.center) to (7.center);
		\draw (7.center) to (6.center);
		\draw (6.center) to (5.center);
		\draw (5.center) to (4.center);
		\draw (0) to (1.center);
		\draw [bend left=15, looseness=1.00] (2.center) to (0);
		\draw [bend left=15, looseness=1.00] (0) to (3.center);
		\draw (15.center) to (16.center);
		\draw (9.center) to (10.center);
		\draw (11.center) to (12.center);
		\draw (13.center) to (14.center);
		\draw (18.center) to (19.center);
		\draw (20.center) to (21.center);
		\draw (22.center) to (23.center);
		\draw (24.center) to (26);
	\end{pgfonlayer}
\end{tikzpicture} = \begin{tikzpicture}
	\begin{pgfonlayer}{nodelayer}
		\node [style=none] (0) at (-3, 3) {};
		\node [style=none] (1) at (-3, 0.5) {};
		\node [style=none] (2) at (-3.25, 0.5) {};
		\node [style=none] (3) at (-2.75, 0.5) {};
		\node [style=none] (4) at (-3.15, 0.4) {};
		\node [style=none] (5) at (-2.85, 0.4) {};
		\node [style=none] (6) at (-3.05, 0.3) {};
		\node [style=none] (7) at (-2.95, 0.3) {};
		\node [style=none] (8) at (-3, -0) {};
		\node [style=none] (9) at (-3, -1) {};
		\node [style=circle, scale=2] (10) at (-3, 0.4) {};
		\node [style=none] (11) at (-3, 3.5) {$F(M(U \oa V))$};
	\end{pgfonlayer}
	\begin{pgfonlayer}{edgelayer}
		\draw (2.center) to (3.center);
		\draw (4.center) to (5.center);
		\draw (6.center) to (7.center);
		\draw (0.center) to (1.center);
		\draw (8.center) to (9.center);
	\end{pgfonlayer}
\end{tikzpicture}
\mbox{ \bf{ [Env.2]}}~~~~~
\begin{tikzpicture}[scale=0.8]
	\begin{pgfonlayer}{nodelayer}
		\node [style=circle] (0) at (0, -0) {$f$};
		\node [style=none] (1) at (0, 1.5) {};
		\node [style=none] (2) at (-0.5, -1) {};
		\node [style=none] (3) at (0.5, -1) {};
		\node [style=none] (4) at (-0.5, -1.75) {};
		\node [style=none] (5) at (0.5, -2.75) {};
		\node [style=none] (6) at (-1, -1) {};
		\node [style=none] (7) at (1, -1) {};
		\node [style=none] (8) at (-1, 0.5) {};
		\node [style=none] (9) at (1, 0.5) {};
		\node [style=none] (10) at (-0.75, -1.75) {};
		\node [style=none] (11) at (-0.25, -1.75) {};
		\node [style=none] (12) at (-0.65, -1.85) {};
		\node [style=none] (13) at (-0.35, -1.85) {};
		\node [style=none] (14) at (-0.55, -1.95) {};
		\node [style=none] (15) at (-0.45, -1.95) {};
		\node [style=none] (16) at (-0.5, -2.25) {};
		\node [style=circle] (17) at (-0.5, -3.25) {$\bot$};
		\node [style=circle, scale=2] (18) at (-0.5, -1.85) {};
		\node [style=none] (19) at (0.75, -0.75) {$F$};
	\end{pgfonlayer}
	\begin{pgfonlayer}{edgelayer}
		\draw (1.center) to (0);
		\draw [bend right=15, looseness=1.00] (0) to (2.center);
		\draw [bend left=15, looseness=1.00] (0) to (3.center);
		\draw (2.center) to (4.center);
		\draw (3.center) to (5.center);
		\draw (8.center) to (6.center);
		\draw (6.center) to (7.center);
		\draw (7.center) to (9.center);
		\draw (9.center) to (8.center);
		\draw (16.center) to (17);
		\draw (10.center) to (11.center);
		\draw (12.center) to (13.center);
		\draw (14.center) to (15.center);
	\end{pgfonlayer}
\end{tikzpicture} = \begin{tikzpicture}[scale=0.8]
	\begin{pgfonlayer}{nodelayer}
		\node [style=circle] (0) at (0, -0) {$g$};
		\node [style=none] (1) at (0, 1.5) {};
		\node [style=none] (2) at (-0.5, -1) {};
		\node [style=none] (3) at (0.5, -1) {};
		\node [style=none] (4) at (-0.5, -1.75) {};
		\node [style=none] (5) at (0.5, -2.75) {};
		\node [style=none] (6) at (-1, -1) {};
		\node [style=none] (7) at (1, -1) {};
		\node [style=none] (8) at (-1, 0.5) {};
		\node [style=none] (9) at (1, 0.5) {};
		\node [style=none] (10) at (-0.75, -1.75) {};
		\node [style=none] (11) at (-0.25, -1.75) {};
		\node [style=none] (12) at (-0.65, -1.85) {};
		\node [style=none] (13) at (-0.35, -1.85) {};
		\node [style=none] (14) at (-0.55, -1.95) {};
		\node [style=none] (15) at (-0.45, -1.95) {};
		\node [style=none] (16) at (-0.5, -2.25) {};
		\node [style=circle] (17) at (-0.5, -3.25) {$\bot$};
		\node [style=circle, scale=2] (18) at (-0.5, -1.85) {};
		\node [style=none] (19) at (0.75, -0.75) {$F$};
	\end{pgfonlayer}
	\begin{pgfonlayer}{edgelayer}
		\draw (1.center) to (0);
		\draw [bend right=15, looseness=1.00] (0) to (2.center);
		\draw [bend left=15, looseness=1.00] (0) to (3.center);
		\draw (2.center) to (4.center);
		\draw (3.center) to (5.center);
		\draw (8.center) to (6.center);
		\draw (6.center) to (7.center);
		\draw (7.center) to (9.center);
		\draw (9.center) to (8.center);
		\draw (16.center) to (17);
		\draw (10.center) to (11.center);
		\draw (12.center) to (13.center);
		\draw (14.center) to (15.center);
	\end{pgfonlayer}
\end{tikzpicture} \Leftrightarrow \begin{tikzpicture}
	\begin{pgfonlayer}{nodelayer}
		\node [style=circle] (0) at (0, -0) {$f$};
		\node [style=none] (1) at (0, 1.5) {};
		\node [style=none] (2) at (-0.5, -1) {};
		\node [style=none] (3) at (0.5, -1) {};
		\node [style=none] (4) at (-0.5, -2.25) {};
		\node [style=none] (5) at (0.5, -2.25) {};
	\end{pgfonlayer}
	\begin{pgfonlayer}{edgelayer}
		\draw (1.center) to (0);
		\draw [bend right=15, looseness=1.00] (0) to (2.center);
		\draw [bend left=15, looseness=1.00] (0) to (3.center);
		\draw (2.center) to (4.center);
		\draw (3.center) to (5.center);
	\end{pgfonlayer}
\end{tikzpicture} \sim \begin{tikzpicture}
	\begin{pgfonlayer}{nodelayer}
		\node [style=circle] (0) at (0, -0) {$g$};
		\node [style=none] (1) at (0, 1.5) {};
		\node [style=none] (2) at (-0.5, -1) {};
		\node [style=none] (3) at (0.5, -1) {};
		\node [style=none] (4) at (-0.5, -2.25) {};
		\node [style=none] (5) at (0.5, -2.25) {};
	\end{pgfonlayer}
	\begin{pgfonlayer}{edgelayer}
		\draw (1.center) to (0);
		\draw [bend right=15, looseness=1.00] (0) to (2.center);
		\draw [bend left=15, looseness=1.00] (0) to (3.center);
		\draw (2.center) to (4.center);
		\draw (3.center) to (5.center);
	\end{pgfonlayer}
\end{tikzpicture}
\]

\begin{definition}
An environment structure $F: \C \to \D$ for a mixed unitary category $M: \U \to \C$ has {\bf purification} if 
\begin{itemize}
\item $F$ is bijective on objects, and 
\item for all $f: F(A) \to F(B) \in \D$, there exists a Kraus map $(f', U): A \to B \in \C$ such that 
\[
\mbox{ \bf{[Env.3]} }~~~~~~
 f = \begin{tikzpicture}
	\begin{pgfonlayer}{nodelayer}
		\node [style=circle] (0) at (0, -0) {$g$};
		\node [style=none] (1) at (0, 1.5) {};
		\node [style=none] (2) at (-0.5, -1) {};
		\node [style=none] (3) at (0.5, -1) {};
		\node [style=none] (4) at (-0.5, -1.75) {};
		\node [style=none] (5) at (0.5, -2.75) {};
		\node [style=none] (6) at (-1, -1) {};
		\node [style=none] (7) at (1, -1) {};
		\node [style=none] (8) at (-1, 0.5) {};
		\node [style=none] (9) at (1, 0.5) {};
		\node [style=none] (10) at (-0.75, -1.75) {};
		\node [style=none] (11) at (-0.25, -1.75) {};
		\node [style=none] (12) at (-0.65, -1.85) {};
		\node [style=none] (13) at (-0.35, -1.85) {};
		\node [style=none] (14) at (-0.55, -1.95) {};
		\node [style=none] (15) at (-0.45, -1.95) {};
		\node [style=none] (16) at (-0.5, -2.25) {};
		\node [style=circle] (17) at (-0.5, -3.25) {$\bot$};
		\node [style=circle, scale=2] (18) at (-0.5, -1.85) {};
		\node [style=none] (19) at (0.75, -0.75) {$F$};
	\end{pgfonlayer}
	\begin{pgfonlayer}{edgelayer}
		\draw (1.center) to (0);
		\draw [bend right=15, looseness=1.00] (0) to (2.center);
		\draw [bend left=15, looseness=1.00] (0) to (3.center);
		\draw (2.center) to (4.center);
		\draw (3.center) to (5.center);
		\draw (8.center) to (6.center);
		\draw (6.center) to (7.center);
		\draw (7.center) to (9.center);
		\draw (9.center) to (8.center);
		\draw (16.center) to (17);
		\draw (10.center) to (11.center);
		\draw (12.center) to (13.center);
		\draw (14.center) to (15.center);
	\end{pgfonlayer}
\end{tikzpicture}
\]
Equationally, 
\[  F(A) \to^{f} F(B) = F(A) \to^{F(f')} F(M(U) \oa B) \to^{n_\oa} M(F(U)) \oa F(B) \to^{\envmap \oa 1} \bot \oa F(B) \to^{u_\oa} F(B) \]
\end{itemize}
\end{definition}

The following lemma shows that every mixed unitary category comes equipped with an environment structure with purification given as follows:

\begin{lemma}
\label{Lemma: Env example}
Every mixed unitary category $M: \U \to \C$ has an environment structure given by 
\[ F: \C \to  \CP^\infty(M: \U \to \C) ; 
\begin{matrix}
\xymatrix{
A \ar[d]_{f} \\
B
} \end{matrix}   \mapsto  \begin{matrix}
\xymatrix{
A \ar[d]^{[(f(u_\oa^L)^{-1} (n_\bot^M)^{-1}, \bot)]} \\
M(\bot) \oa B
}
\end{matrix}
\]  and \[ \text{ for all $U \in \U$, } ~~~ \envmap: M(U) \to \bot \in \CP^\infty(M: \U \to \C) := [((u_\oa^R)^{-1}, U)] \] Moreover, this environment structure has purification. 
\end{lemma}

\begin{proof}
$F$ is functorial since $f(u_\oa^L)^{-1} ((n_\bot^M)^{-1} \oa 1) \sim f \sim (u_\oa^L)^{-1} (n_\bot^M)^{-1} (1 \oa f) $. Define $F_\ox = F_\oa := F$. Note that, $F$ is a strict monoidal functor and an isomix functor.
In order to prove that $F: \C \to \D$ with $\envmap$ satisfy axioms for environment structures, the properties of isomix functor and Lemma \ref{Lemma: unitary equivalence} are used.

 To prove that this environment structure has purification, consider any map $[(f, U)]: A \to B \in \CP^\infty(M: \U \to \C)$.  Then, there exists a Kraus map $(f,U): A \to B \in \C$. Then, the map in equation {\bf [Env. 3] } is drawn as follows:

\[ 
\begin{tikzpicture}[scale=1.5]
\draw (0.75,1) -- (0.75,1.5);
\draw (0.25,0.5) -- (1.25, 0.5) -- (1.25, 1) -- (0.25, 1) -- (0.25, 0.5);
\draw (0.75,0.75) node{F(f)};
\draw (0.55,0.5) -- (0.55,0);
\draw (0.95,0.5) -- (0.95,0);
\draw (0.40, 0) -- (0.70, 0);
\draw (0.45, -0.1) -- (0.65, -0.1);
\draw (0.50, -0.2) -- (0.60, -0.2);
\end{tikzpicture} = \left[
\begin{tikzpicture}
	\begin{pgfonlayer}{nodelayer}
		\node [style=circle, scale=0.5] (0) at (-2.25, 0.25) {};
		\node [style=circle, scale=1.25] (1) at (-1.5, -0.5) {};
		\node [style=none] (100) at (-1.5, -0.5) {$\bot$};
		\node [style=circle, scale=0.5] (2) at (0, 0.25) {};
		\node [style=circle, scale=1.25] (3) at (-0.75, -0.5) {};
		\node [style=none] (30) at (-0.75, -0.5) {$\bot$};
		\node [style=oa] (4) at (-2.25, -1.75) {};
		\node [style=oa] (5) at (0, -1.75) {};
		\node [style=none] (6) at (0, -3) {};
		\node [style=none] (7) at (-2.5, 0.5) {$U$};
		\node [style=none] (8) at (0.25, 0.5) {$B$};
		\node [style=none] (9) at (0, -3.25) {$\bot \oa B$};
		\node [style=circle, scale=1.55] (10) at (-1, 1.5) {};
		\node [style=none] (101) at (-1, 1.5) {$f$};
		\node [style=none] (11) at (-1, 2.5) {};
		\node [style=circle, scale=0.5] (12) at (-1, 2) {};
		\node [style=circle, scale=1.25] (13) at (-3.25, -1.75) {};
		\node [style=none] (31) at (-3.25, -1.75) {$\bot$};
		\node [style=oa] (14) at (-2.75, -2.5) {};
		\node [style=none] (15) at (-2.75, -3) {};
	\end{pgfonlayer}
	\begin{pgfonlayer}{edgelayer}
		\draw [dotted, bend left, looseness=1.00] (0) to (1);
		\draw [dotted, bend right, looseness=1.25] (2) to (3);
		\draw [in=140, out=-90, looseness=1.25] (1) to (5);
		\draw (2) to (5);
		\draw [in=40, out=-90, looseness=1.00] (3) to (4);
		\draw (0) to (4);
		\draw (5) to (6.center);
		\draw [in=-150, out=90, looseness=1.00] (0) to (10);
		\draw [in=90, out=-30, looseness=1.00] (10) to (2);
		\draw (11.center) to (12);
		\draw (12) to (10);
		\draw [dotted, bend right, looseness=1.25] (12) to (13);
		\draw (14) to (15.center);
		\draw [bend right, looseness=1.00] (13) to (14);
		\draw [bend left, looseness=1.00] (4) to (14);
	\end{pgfonlayer}
\end{tikzpicture}
\right] \]  Because, $\begin{tikzpicture}
	\begin{pgfonlayer}{nodelayer}
		\node [style=circle, scale=0.5] (0) at (-2.25, 0.25) {};
		\node [style=circle, scale=1.25] (1) at (-1.5, -0.5) {};
		\node [style=none] (100) at (-1.5, -0.5) {$\bot$};
		\node [style=circle, scale=0.5] (2) at (0, 0.25) {};
		\node [style=circle, scale=1.25] (3) at (-0.75, -0.5) {};
		\node [style=none] (30) at (-0.75, -0.5) {$\bot$};
		\node [style=oa] (4) at (-2.25, -1.75) {};
		\node [style=oa] (5) at (0, -1.75) {};
		\node [style=none] (6) at (0, -3) {};
		\node [style=none] (7) at (-2.5, 0.5) {$U$};
		\node [style=none] (8) at (0.25, 0.5) {$B$};
		\node [style=none] (9) at (0, -3.25) {$\bot \oa B$};
		\node [style=circle, scale=1.55] (10) at (-1, 1.5) {};
		\node [style=none] (101) at (-1, 1.5) {$f$};
		\node [style=none] (11) at (-1, 2.5) {};
		\node [style=circle, scale=0.5] (12) at (-1, 2) {};
		\node [style=circle, scale=1.25] (13) at (-3.25, -1.75) {};
		\node [style=none] (31) at (-3.25, -1.75) {$\bot$};
		\node [style=oa] (14) at (-2.75, -2.5) {};
		\node [style=none] (15) at (-2.75, -3) {};
	\end{pgfonlayer}
	\begin{pgfonlayer}{edgelayer}
		\draw [dotted, bend left, looseness=1.00] (0) to (1);
		\draw [dotted, bend right, looseness=1.25] (2) to (3);
		\draw [in=140, out=-90, looseness=1.25] (1) to (5);
		\draw (2) to (5);
		\draw [in=40, out=-90, looseness=1.00] (3) to (4);
		\draw (0) to (4);
		\draw (5) to (6.center);
		\draw [in=-150, out=90, looseness=1.00] (0) to (10);
		\draw [in=90, out=-30, looseness=1.00] (10) to (2);
		\draw (11.center) to (12);
		\draw (12) to (10);
		\draw [dotted, bend right, looseness=1.25] (12) to (13);
		\draw (14) to (15.center);
		\draw [bend right, looseness=1.00] (13) to (14);
		\draw [bend left, looseness=1.00] (4) to (14);
	\end{pgfonlayer}
\end{tikzpicture} \sim 
\begin{tikzpicture} 
	\begin{pgfonlayer}{nodelayer}
		\node [style=circle] (0) at (-1, 0.25) {$f$};
		\node [style=none] (1) at (-1.5, -1.5) {};
		\node [style=none] (2) at (-0.5, -1.5) {};
		\node [style=none] (3) at (-1, 2.75) {};
	\end{pgfonlayer}
	\begin{pgfonlayer}{edgelayer}
		\draw [style=none] (3.center) to (0);
		\draw [style=none, in=90, out=-135, looseness=1.00] (0) to (1.center);
		\draw [style=none, in=90, out=-45, looseness=1.00] (0) to (2.center);
	\end{pgfonlayer}
\end{tikzpicture}
$, the environment structure $(Q: \C \to \CP^\infty(M: \U \to \C), \envmap)$ has purification.
\end{proof}

The following are the environment structures for our running examples:
\begin{itemize}
\item Consider the MUC, $\R^* \subset \C$. Then, \[ (\R^* \to^{Q} \CP^\infty( \R \subset \C), \envmap_r: r \to 1)\] is an environment structure where, $ \envmap_r := (=, 1/r) : r \to 1  $
\item Consider the MUC, ${\sf Mat}_\C \to \FMat_\C$. Then, 
\[ {\sf Mat}_\C \to^{Q} \CP^\infty({\sf Mat}_\C \subset \FMat_\C) \] is an environment structure where, $\envmap_{\C^n}: \C^n \to \C ; \rho \mapsto Tr(\rho) $.
\end{itemize}

\subsection{Characterizing the $\CP$-infinity construction}

In this section, we show that any environment structure with purification is initial in the category of environment structures.  Given Lemma \ref{Lemma: Env example}, this shows that any environment structure over a $M:\U \to \C$ with purification is isomorphic to $\CP^\infty(M: \U \to \C)$  Thus, environment structures with purification captures the abstract structure of $\CP^\infty( M: \U \to \C)$.

\begin{definition}
Let $M: \U \to \C$ be a mixed unitary category. Define a category $\mathsf{Env}(M: \U \to \C)$ as follows:
\begin{description}
\item[Objects:] Environment structures for $M: \U \to \C$
\item[Arrows:] Suppose $D: \C \to \D$ with $\envmap$, and $D': \C \to \D'$ with $\anotherenvmap{1.8}$ are two environment structures. Then,  a morphism of environment structures is a strict isomix functor $F: \D \to \D'$ such that 
\begin{itemize}
 \item $DF = D'$
 \item $F(\envmap) = \anotherenvmap{1.8}$
\end{itemize}
\item[Identity arrows:] Identity functor
\item[Composition:] Linear functor composition
\end{description}
\end{definition}

\begin{lemma}
\label{Lemma: Env initial}
Let $M: \U \to \C$ be a mixed unitary category. Suppose $D: \C \to \D$ with $\envmap$ is an environment structure with purification, then it is initial in ${\sf Env}(M: \U \to \C)$.
\end{lemma}
\begin{proof}
Suppose $(D': \X \rightarrow \Y', \anotherenvmap{1.5}) \in \mathsf{Env(\C)}$. We show that there is a unique strict isomix functor $F: \Y \rightarrow \Y'$ such that $DF = D'$ and $F(\envmap) = \anotherenvmap{1.5}$.

Define $F: \D \rightarrow \D'$ as follows:
\begin{itemize}
\item  Since $(D, \envmap)$ has purification, $D: \C \to \D$ is bijective on objects. Then for all $A \in \D$, $A = D(X)$ for a unique $X \in \C$. Then,
\[ F(A) := D'(X) \]
\item Let $f: A \rightarrow B \in \D$. Since $(D: \C \to \D, \envmap)$ has purification,   
\begin{align*}
\begin{tikzpicture}
	\begin{pgfonlayer}{nodelayer}
		\node [style=circle] (0) at (0, 1) {$f$};
		\node [style=none] (1) at (0, 2) {};
		\node [style=none] (2) at (0, -0) {};
		\node [style=none] (3) at (0.2, 2.2) {$A$};
		\node [style=none] (4) at (0.2, -0.2) {$B$};
	\end{pgfonlayer}
	\begin{pgfonlayer}{edgelayer}
		\draw (1.center) to (0);
		\draw (0) to (2.center);
	\end{pgfonlayer}
\end{tikzpicture} = \begin{tikzpicture}
	\begin{pgfonlayer}{nodelayer}
		\node [style=circle, scale=3] (0) at (0, 1) {};
		\node [style=none] (5) at (0, 1) {$D(f')$};
		\node [style=none] (1) at (0, 2) {};
		\node [style=none] (2) at (-0.6, -0) {};
		\node [style=none] (3) at (0.6, -0) {};
		\node [style=none] (4) at (0.9, -0.3) {$D(Y)$};
		\node [style=none] (6) at (0.2, 2.2) {$D(X)$};
	\end{pgfonlayer}
	\begin{pgfonlayer}{edgelayer}
		\draw (1.center) to (0);
		\draw [bend right, looseness=1.] (0) to (2.center);
		\draw [bend left, looseness=1] (0) to (3.center);
\draw (-0.85, 0) -- (-0.35, 0);
\draw (-0.75, -0.1) -- (-0.45, -0.1);
\draw (-0.65, -0.2) -- (-0.55, -0.2);
	\end{pgfonlayer}
\end{tikzpicture}  
 \xmapsto{F} ~~~ \begin{tikzpicture}
	\begin{pgfonlayer}{nodelayer}
		\node [style=circle, scale=3] (0) at (0, 1) {};
		\node [style=none] (5) at (0, 1) {$D'({f'})$};
		\node [style=none] (1) at (0, 2) {};
		\node [style=none] (2) at (-0.6, -0) {};
		\node [style=none] (3) at (0.6, -0) {};
		\node [style=none] (4) at (0.9, -0.3) {$D'(Y)$};
		\node [style=none] (6) at (0.2, 2.2) {$D'(X)$};
	\end{pgfonlayer}
	\begin{pgfonlayer}{edgelayer}
		\draw (1.center) to (0);
		\draw [bend right, looseness=1.] (0) to (2.center);
		\draw [bend left, looseness=1] (0) to (3.center);
		\draw (-0.85, 0) -- (-0.35, 0);
		\draw  [bend right=90, looseness=1.25] (-0.85, 0) to (-0.35, 0);
	\end{pgfonlayer}
\end{tikzpicture} 
\end{align*}
where $F(\envmap) = \anotherenvmap{1.5}$.
\end{itemize}

This fixes the definition of $F$. To prove that $F$ is well-defined on arrows we need to show that $f=g \Rightarrow F(f) = F(g)$. Since, $(D, \envmap)$ has purification, let
\[
\begin{tikzpicture}
	\begin{pgfonlayer}{nodelayer}
		\node [style=circle] (0) at (0, 1) {$f$};
		\node [style=none] (1) at (0, 2) {};
		\node [style=none] (2) at (0, -0) {};
	\end{pgfonlayer}
	\begin{pgfonlayer}{edgelayer}
		\draw (1.center) to (0);
		\draw (0) to (2.center);
	\end{pgfonlayer}
\end{tikzpicture} = \begin{tikzpicture}
	\begin{pgfonlayer}{nodelayer}
		\node [style=circle, scale=3] (0) at (0, 1) {};
		\node [style=none] (5) at (0, 1) {$D(\overline{f})$};
		\node [style=none] (1) at (0, 2) {};
		\node [style=none] (2) at (-0.6, -0) {};
		\node [style=none] (3) at (0.6, -0) {};
	\end{pgfonlayer}
	\begin{pgfonlayer}{edgelayer}
		\draw (1.center) to (0);
		\draw [bend right, looseness=1.] (0) to (2.center);
		\draw [bend left, looseness=1] (0) to (3.center);
\draw (-0.85, 0) -- (-0.35, 0);
\draw (-0.75, -0.1) -- (-0.45, -0.1);
\draw (-0.65, -0.2) -- (-0.55, -0.2);
	\end{pgfonlayer}
\end{tikzpicture} ~~~~~~~ \begin{tikzpicture}
	\begin{pgfonlayer}{nodelayer}
		\node [style=circle] (0) at (0, 1) {$g$};
		\node [style=none] (1) at (0, 2) {};
		\node [style=none] (2) at (0, -0) {};
	\end{pgfonlayer}
	\begin{pgfonlayer}{edgelayer}
		\draw (1.center) to (0);
		\draw (0) to (2.center);
	\end{pgfonlayer}
\end{tikzpicture} = \begin{tikzpicture}
	\begin{pgfonlayer}{nodelayer}
		\node [style=circle, scale=3] (0) at (0, 1) {};
		\node [style=none] (5) at (0, 1) {$D(\overline{g})$};
		\node [style=none] (1) at (0, 2) {};
		\node [style=none] (2) at (-0.6, -0) {};
		\node [style=none] (3) at (0.6, -0) {};
	\end{pgfonlayer}
	\begin{pgfonlayer}{edgelayer}
		\draw (1.center) to (0);
		\draw [bend right, looseness=1.] (0) to (2.center);
		\draw [bend left, looseness=1] (0) to (3.center);
\draw (-0.85, 0) -- (-0.35, 0);
\draw (-0.75, -0.1) -- (-0.45, -0.1);
\draw (-0.65, -0.2) -- (-0.55, -0.2);
	\end{pgfonlayer}
\end{tikzpicture} 
\]
Then,
\[
 \begin{tikzpicture}
	\begin{pgfonlayer}{nodelayer}
		\node [style=circle, scale=3] (0) at (0, 1) {};
		\node [style=none] (5) at (0, 1) {$D(\overline{f})$};
		\node [style=none] (1) at (0, 2) {};
		\node [style=none] (2) at (-0.6, -0) {};
		\node [style=none] (3) at (0.6, -0) {};
	\end{pgfonlayer}
	\begin{pgfonlayer}{edgelayer}
		\draw (1.center) to (0);
		\draw [bend right, looseness=1.] (0) to (2.center);
		\draw [bend left, looseness=1] (0) to (3.center);
\draw (-0.85, 0) -- (-0.35, 0);
\draw (-0.75, -0.1) -- (-0.45, -0.1);
\draw (-0.65, -0.2) -- (-0.55, -0.2);
	\end{pgfonlayer}
\end{tikzpicture}  =  \begin{tikzpicture}
	\begin{pgfonlayer}{nodelayer}
		\node [style=circle, scale=3] (0) at (0, 1) {};
		\node [style=none] (5) at (0, 1) {$D(\overline{g})$};
		\node [style=none] (1) at (0, 2) {};
		\node [style=none] (2) at (-0.6, -0) {};
		\node [style=none] (3) at (0.6, -0) {};
	\end{pgfonlayer}
	\begin{pgfonlayer}{edgelayer}
		\draw (1.center) to (0);
		\draw [bend right, looseness=1.] (0) to (2.center);
		\draw [bend left, looseness=1] (0) to (3.center);
\draw (-0.85, 0) -- (-0.35, 0);
\draw (-0.75, -0.1) -- (-0.45, -0.1);
\draw (-0.65, -0.2) -- (-0.55, -0.2);
	\end{pgfonlayer}
\end{tikzpicture}  \Leftrightarrow  \begin{tikzpicture}
	\begin{pgfonlayer}{nodelayer}
		\node [style=circle, scale=2] (0) at (0, 1) {};
		\node [style=none] (5) at (0, 1) {$\overline{f}$};
		\node [style=none] (1) at (0, 2) {};
		\node [style=none] (2) at (-0.6, -0) {};
		\node [style=none] (3) at (0.6, -0) {};
	\end{pgfonlayer}
	\begin{pgfonlayer}{edgelayer}
		\draw (1.center) to (0);
		\draw [bend right, looseness=1.] (0) to (2.center);
		\draw [bend left, looseness=1] (0) to (3.center);
	\end{pgfonlayer}
\end{tikzpicture} = \begin{tikzpicture}
	\begin{pgfonlayer}{nodelayer}
		\node [style=circle, scale=2] (0) at (0, 1) {};
		\node [style=none] (5) at (0, 1) {$\overline{g}$};
		\node [style=none] (1) at (0, 2) {};
		\node [style=none] (2) at (-0.6, -0) {};
		\node [style=none] (3) at (0.6, -0) {};
	\end{pgfonlayer}
	\begin{pgfonlayer}{edgelayer}
		\draw (1.center) to (0);
		\draw [bend right, looseness=1.] (0) to (2.center);
		\draw [bend left, looseness=1] (0) to (3.center);
	\end{pgfonlayer}
\end{tikzpicture} \Leftrightarrow \begin{tikzpicture}
	\begin{pgfonlayer}{nodelayer}
		\node [style=circle, scale=3] (0) at (0, 1) {};
		\node [style=none] (5) at (0, 1) {$D'(\overline{f})$};
		\node [style=none] (1) at (0, 2) {};
		\node [style=none] (2) at (-0.6, -0) {};
		\node [style=none] (3) at (0.6, -0) {};
	\end{pgfonlayer}
	\begin{pgfonlayer}{edgelayer}
		\draw (1.center) to (0);
		\draw [bend right, looseness=1.] (0) to (2.center);
		\draw [bend left, looseness=1] (0) to (3.center);
		\draw (-0.85, 0) -- (-0.35, 0);
		\draw  [bend right=90, looseness=1.25] (-0.85, 0) to (-0.35, 0);
	\end{pgfonlayer}
\end{tikzpicture}  =  \begin{tikzpicture}
	\begin{pgfonlayer}{nodelayer}
		\node [style=circle, scale=3] (0) at (0, 1) {};
		\node [style=none] (5) at (0, 1) {$D'(\overline{g})$};
		\node [style=none] (1) at (0, 2) {};
		\node [style=none] (2) at (-0.6, -0) {};
		\node [style=none] (3) at (0.6, -0) {};
	\end{pgfonlayer}
	\begin{pgfonlayer}{edgelayer}
		\draw (1.center) to (0);
		\draw [bend right, looseness=1.] (0) to (2.center);
		\draw [bend left, looseness=1] (0) to (3.center);
		\draw (-0.85, 0) -- (-0.35, 0);
		\draw  [bend right=90, looseness=1.25] (-0.85, 0) to (-0.35, 0);
	\end{pgfonlayer}
\end{tikzpicture} 
\]

$F: \D \to \D'$ preserves identity: \[ F(1_A) = F(1_{D(X)}) = F(D(1_X)) = D'(1_X) = 1_{D'(X)} = 1_{F(D(X))} = 1_{F(A)} \]

$F: \D \to \D'$ preserves compostion:
\[ F \left( 
\begin{tikzpicture}
	\begin{pgfonlayer}{nodelayer}
		\node [style=circle] (0) at (0, 2) {$f$};
		\node [style=circle] (1) at (0, 0.75) {$g$};
		\node [style=none] (2) at (0, -0.5) {};
		\node [style=none] (3) at (0, 3) {};
	\end{pgfonlayer}
	\begin{pgfonlayer}{edgelayer}
		\draw (3.center) to (0);
		\draw (0) to (1);
		\draw (1) to (2.center);
	\end{pgfonlayer}
\end{tikzpicture}
\right ) = 
F \left(\begin{tikzpicture} 
	\begin{pgfonlayer}{nodelayer}
		\node [style=circle, scale=2.5] (0) at (-1, 2) {};
		\node [style=circle, scale=2.5] (1) at (-1, 0.5) {};
		\node [style=none] (100) at (-1, 0.5) {$D(\overline{g})$};
		\node [style=none] (101) at (-1, 2) {$D(\overline{f})$};
		\node [style=none] (2) at (-1, -0.5) {};
		\node [style=none] (3) at (-1, 3) {};
		\node [style=none] (4) at (-2.75, -0.5) {};
		\node [style=none] (5) at (-1.75, -0.5) {};

		\node [style=none] (6) at (-3, -0.5) {};
		\node [style=none] (8) at (-2.9, -0.6) {};
		\node [style=none] (9) at (-2.8, -0.7) {};
		\node [style=none] (7) at (-2.5, -0.5) {};
		\node [style=none] (10) at (-2.6, -0.6) {};
		\node [style=none] (11) at (-2.7, -0.7) {};
				
		\node [style=none] (13) at (-1.5, -0.5) {};
		\node [style=none] (15) at (-1.6, -0.6) {};
		\node [style=none] (17) at (-1.7, -0.7) {};
		\node [style=none] (12) at (-2, -0.5) {};
		\node [style=none] (14) at (-1.9, -0.6) {};
		\node [style=none] (16) at (-1.8, -0.7) {};
	\end{pgfonlayer}
	\begin{pgfonlayer}{edgelayer}
		\draw (3.center) to (0);
		\draw (0) to (1);
		\draw (1) to (2.center);
		\draw [bend right=45, looseness=0.75] (1) to (5.center);
		\draw [bend right, looseness=1.00] (0) to (4.center);
		\draw (6.center) to (7.center);
		\draw (8.center) to (10.center);
		\draw (9.center) to (11.center);
		\draw (16.center) to (17.center);
		\draw (14.center) to (15.center);
		\draw (12.center) to (13.center);
	\end{pgfonlayer}
\end{tikzpicture} \right) \stackrel{\tiny{\bf{Env.1a} }}{=} F \left( 
\begin{tikzpicture} 
	\begin{pgfonlayer}{nodelayer}
		\node [style=circle, scale=2.8] (0) at (-1, 2) {};
		\node [style=circle, scale=2.8] (1) at (-1, 0.5) {};
		\node [style=none] (2) at (-1, -0.5) {};
		\node [style=none] (3) at (-1, 3) {};
		\node [style=none] (4) at (-2, -0.25) {};
		\node [style=none] (5) at (-1.75, -0.25) {};
		\node [style=none] (6) at (-1, 2) {$D'(\overline{f})$};
		\node [style=none] (7) at (-1, 0.5) {$D'(\overline{g})$};
		\node [style=none] (8) at (-2.25, -0.25) {};
		\node [style=none] (9) at (-1.5, -0.25) {};
		\node [style=none] (10) at (-2.1, -0.35) {};
		\node [style=none] (11) at (-1.65, -0.35) {};	
		\node [style=none] (12) at (-1.95, -0.45) {};
		\node [style=none] (13) at (-1.85, -0.45) {};		
	\end{pgfonlayer}
	\begin{pgfonlayer}{edgelayer}
		\draw (3.center) to (0);
		\draw (0) to (1);
		\draw (1) to (2.center);
		\draw [bend right=45, looseness=0.75] (1) to (5.center);
		\draw [bend right, looseness=1.00] (0) to (4.center);
		\draw (8.center) to (9.center);
		\draw (10.center) to (11.center);
		\draw (12.center) to (13.center);
	\end{pgfonlayer}
\end{tikzpicture}
 \right) :=
\begin{tikzpicture} 
	\begin{pgfonlayer}{nodelayer}
		\node [style=circle, scale=2.8] (0) at (-1, 2) {};
		\node [style=circle, scale=2.8] (1) at (-1, 0.5) {};
		\node [style=none] (2) at (-1, -0.5) {};
		\node [style=none] (3) at (-1, 3) {};
		\node [style=none] (4) at (-2, -0.25) {};
		\node [style=none] (5) at (-1.75, -0.25) {};
		\node [style=none] (6) at (-1, 2) {$D'(\overline{f})$};
		\node [style=none] (7) at (-1, 0.5) {$D'(\overline{g})$};
		\node [style=none] (8) at (-2.25, -0.25) {};
		\node [style=none] (9) at (-1.5, -0.25) {};
	\end{pgfonlayer}
	\begin{pgfonlayer}{edgelayer}
		\draw (3.center) to (0);
		\draw (0) to (1);
		\draw (1) to (2.center);
		\draw [bend right=45, looseness=0.75] (1) to (5.center);
		\draw [bend right, looseness=1.00] (0) to (4.center);
		\draw (8.center) to (9.center);		
		\draw [bend right = 90, looseness=1.00] (8.center) to (9.center);
	\end{pgfonlayer}
\end{tikzpicture}
\stackrel{\tiny{\bf{Env.1a} }}{=} \begin{tikzpicture} 
	\begin{pgfonlayer}{nodelayer}
		\node [style=circle, scale=2.8] (0) at (-1, 2) {};
		\node [style=circle, scale=2.8] (1) at (-1, 0.5) {};
		\node [style=none] (100) at (-1, 0.5) {$D'(\overline{g})$};
		\node [style=none] (101) at (-1, 2) {$D'(\overline{f})$};
		\node [style=none] (2) at (-1, -0.5) {};
		\node [style=none] (3) at (-1, 3) {};
		\node [style=none] (4) at (-2.75, -0.5) {};
		\node [style=none] (5) at (-1.75, -0.5) {};

		\node [style=none] (6) at (-3, -0.5) {};
		\node [style=none] (7) at (-2.5, -0.5) {};
				
		\node [style=none] (13) at (-1.5, -0.5) {};
		\node [style=none] (12) at (-2, -0.5) {};
	\end{pgfonlayer}
	\begin{pgfonlayer}{edgelayer}
		\draw (3.center) to (0);
		\draw (0) to (1);
		\draw (1) to (2.center);
		\draw [bend right=45, looseness=0.75] (1) to (5.center);
		\draw [bend right, looseness=1.00] (0) to (4.center);
		\draw (6.center) to (7.center);
		\draw (12.center) to (13.center);
		\draw [bend right=90, looseness=1.20] (6.center) to (7.center);
		\draw [bend right=90, looseness=1.20] (12.center) to (13.center);
	\end{pgfonlayer}
\end{tikzpicture} = \begin{tikzpicture}
	\begin{pgfonlayer}{nodelayer}
		\node [style=circle, scale=2.8] (0) at (0, 2) {};
		\node [style=none] (100) at (0, 2) {$F(f)$};
		\node [style=circle, scale=2.8] (1) at (0, 0.75) {};
		\node [style=none] (101) at (0, 0.75) {$F(g)$};
		\node [style=none] (2) at (0, -0.5) {};
		\node [style=none] (3) at (0, 3) {};
	\end{pgfonlayer}
	\begin{pgfonlayer}{edgelayer}
		\draw (3.center) to (0);
		\draw (0) to (1);
		\draw (1) to (2.center);
	\end{pgfonlayer}
\end{tikzpicture}
\]

$F: \D \to \D'$ is strict monoidal in $\otimes$:
\[ 
F \left(
\begin{tikzpicture}
	\begin{pgfonlayer}{nodelayer}
		\node [style=circle] (0) at (0, 1) {$f$};
		\node [style=none] (1) at (0, 2) {};
		\node [style=none] (2) at (0, -0) {};
	\end{pgfonlayer}
	\begin{pgfonlayer}{edgelayer}
		\draw (1.center) to (0);
		\draw (0) to (2.center);
	\end{pgfonlayer}
\end{tikzpicture} \ox \begin{tikzpicture}
	\begin{pgfonlayer}{nodelayer}
		\node [style=circle] (0) at (0, 1) {$g$};
		\node [style=none] (1) at (0, 2) {};
		\node [style=none] (2) at (0, -0) {};
	\end{pgfonlayer}
	\begin{pgfonlayer}{edgelayer}
		\draw (1.center) to (0);
		\draw (0) to (2.center);
	\end{pgfonlayer}
\end{tikzpicture} \right) = F \left( \begin{tikzpicture}
	\begin{pgfonlayer}{nodelayer}
		\node [style=circle, scale=3] (0) at (0, 1) {};
		\node [style=none] (5) at (0, 1) {$D(\overline{f})$};
		\node [style=none] (1) at (0, 2) {};
		\node [style=none] (2) at (-0.6, -0) {};
		\node [style=none] (3) at (0.6, -0) {};
	\end{pgfonlayer}
	\begin{pgfonlayer}{edgelayer}
		\draw (1.center) to (0);
		\draw [bend right, looseness=1.] (0) to (2.center);
		\draw [bend left, looseness=1] (0) to (3.center);
\draw (-0.85, 0) -- (-0.35, 0);
\draw (-0.75, -0.1) -- (-0.45, -0.1);
\draw (-0.65, -0.2) -- (-0.55, -0.2);
	\end{pgfonlayer}
\end{tikzpicture} \ox \begin{tikzpicture}
	\begin{pgfonlayer}{nodelayer}
		\node [style=circle, scale=3] (0) at (0, 1) {};
		\node [style=none] (5) at (0, 1) {$D(\overline{g})$};
		\node [style=none] (1) at (0, 2) {};
		\node [style=none] (2) at (-0.6, -0) {};
		\node [style=none] (3) at (0.6, -0) {};
	\end{pgfonlayer}
	\begin{pgfonlayer}{edgelayer}
		\draw (1.center) to (0);
		\draw [bend right, looseness=1.] (0) to (2.center);
		\draw [bend left, looseness=1] (0) to (3.center);
\draw (-0.85, 0) -- (-0.35, 0);
\draw (-0.75, -0.1) -- (-0.45, -0.1);
\draw (-0.65, -0.2) -- (-0.55, -0.2);
	\end{pgfonlayer}
\end{tikzpicture} \right) = F\left( 
\begin{tikzpicture}
	\begin{pgfonlayer}{nodelayer}
		\node [style=circle, scale=2.8] (0) at (-2, 1) {};
		\node [style=none] (1) at (-2, 2.25) {};
		\node [style=circle, scale=2.8] (2) at (-0.5, 1) {};
		\node [style=none] (3) at (-0.5, 2.25) {};
		\node [style=none] (4) at (-2.75, -0.75) {};
		\node [style=none] (5) at (-2.25, -0.75) {};
		\node [style=none] (6) at (-1, -1) {};
		\node [style=none] (7) at (0, -1) {};
		\node [style=none] (8) at (-2, 1) {$D(\overline{f})$};
		\node [style=none] (9) at (-0.5, 1) {$D(\overline{g})$};
		\node [style=none] (10) at (-3, -0.75) {};
		\node [style=none] (11) at (-2.85, -0.85) {};
		\node [style=none] (12) at (-2.7, -0.95) {};
		\node [style=none] (13) at (-2, -0.75) {};
		\node [style=none] (14) at (-2.15, -0.85) {};
		\node [style=none] (15) at (-2.3, -0.95) {};
	\end{pgfonlayer}
	\begin{pgfonlayer}{edgelayer}
		\draw [bend left, looseness=1.25] (0) to (6.center);
		\draw [bend left=15, looseness=1.00] (2) to (7.center);
		\draw [bend right=15, looseness=1.00] (2) to (5.center);
		\draw [bend right=15, looseness=1.00] (0) to (4.center);
		\draw (0) to (1.center);
		\draw (3.center) to (2);
		\draw (10.center) to (13.center);
		\draw (11.center) to (14.center);
		\draw (12.center) to (15.center);
	\end{pgfonlayer}
\end{tikzpicture}
\right)=  
\begin{tikzpicture}
	\begin{pgfonlayer}{nodelayer}
		\node [style=circle, scale=3] (0) at (-2, 1) {};
		\node [style=none] (1) at (-2, 2.25) {};
		\node [style=circle, scale=3] (2) at (-0.5, 1) {};
		\node [style=none] (3) at (-0.5, 2.25) {};
		\node [style=none] (4) at (-2.75, -0.75) {};
		\node [style=none] (5) at (-2.25, -0.75) {};
		\node [style=none] (6) at (-1, -1) {};
		\node [style=none] (7) at (0, -1) {};
		\node [style=none] (8) at (-2, 1) {$D'(\overline{f})$};
		\node [style=none] (9) at (-0.5, 1) {$D'(\overline{g})$};
		\node [style=none] (10) at (-3, -0.75) {};
		\node [style=none] (13) at (-2, -0.75) {};
	\end{pgfonlayer}
	\begin{pgfonlayer}{edgelayer}
		\draw [bend left, looseness=1.25] (0) to (6.center);
		\draw [bend left=15, looseness=1.00] (2) to (7.center);
		\draw [bend right=15, looseness=1.00] (2) to (5.center);
		\draw [bend right=15, looseness=1.00] (0) to (4.center);
		\draw (0) to (1.center);
		\draw (3.center) to (2);
		\draw (10.center) to (13.center);
		\draw [bend right=90, looseness=1.25] (10.center) to (13.center);
	\end{pgfonlayer}
\end{tikzpicture}
= \begin{tikzpicture}
	\begin{pgfonlayer}{nodelayer}
		\node [style=circle, scale=3] (0) at (0, 1) {};
		\node [style=none] (5) at (0, 1) {$D'(\overline{f})$};
		\node [style=none] (1) at (0, 2) {};
		\node [style=none] (2) at (-0.6, -0) {};
		\node [style=none] (3) at (0.6, -0) {};
	\end{pgfonlayer}
	\begin{pgfonlayer}{edgelayer}
		\draw (1.center) to (0);
		\draw [bend right, looseness=1.] (0) to (2.center);
		\draw [bend left, looseness=1] (0) to (3.center);
		\draw (-0.85, 0) -- (-0.35, 0);
		\draw  [bend right=90, looseness=1.25] (-0.85, 0) to (-0.35, 0);
	\end{pgfonlayer}
\end{tikzpicture}  \ox \begin{tikzpicture}
	\begin{pgfonlayer}{nodelayer}
		\node [style=circle, scale=3] (0) at (0, 1) {};
		\node [style=none] (5) at (0, 1) {$D'(\overline{g})$};
		\node [style=none] (1) at (0, 2) {};
		\node [style=none] (2) at (-0.6, -0) {};
		\node [style=none] (3) at (0.6, -0) {};
	\end{pgfonlayer}
	\begin{pgfonlayer}{edgelayer}
		\draw (1.center) to (0);
		\draw [bend right, looseness=1.] (0) to (2.center);
		\draw [bend left, looseness=1] (0) to (3.center);
		\draw (-0.85, 0) -- (-0.35, 0);
		\draw  [bend right=90, looseness=1.25] (-0.85, 0) to (-0.35, 0);
	\end{pgfonlayer}
\end{tikzpicture} 
\]
\[ \text{and, } F((u_\ox^L)_A) = F((u_\ox^L)_{D(X)}) = F(D((u_\ox^L)_X)) = D'((u_\ox^L)_X) = (u_\ox^L)_{D'(X)} = (u_\ox^L)_{F(D(X))} = (u_\ox^L)_{F(A)}\]

Smiliarly, it can be proved that $F$ is strict comonoidal in $\oa$. 

Define $F_\ox = F_\oa := F$ and linear strengths to be identity maps. Thus, $F$ is a unique strict Frobenius functor. $F$ is an isomix functor because D and D' preserve the mix map $\m$ on the nose.
\end{proof}

\begin{corollary}
Suppose $D: \C \to \D$ with $\envmap$ is an environment structure with purification for $M:\U \to \C$. Then, $\D \simeq \CP^\infty(M: \U \to \C)$.
\end{corollary}
\begin{proof}
By Lemma \ref{Lemma: Env initial}, $D: \C \to \D$ with $\envmap$ is initial in $\mathsf{Env}(M: \U \to \C)$. By Lemma \ref{Lemma: Env example},  $F: \C \to \CP^\infty(M: \U \to \C)$ with $[(u_\oa^R)^{-1},U)]$ is an environment structure for $M: \U \to \C$ which has purification, hence it is also an initial object in $\mathsf{Env}(M: \U \to \C)$. Since, initial objects of a category are isomorphic, there exists a strict isomix functor $\D \xrightarrow{F} \CP^\infty(M: \U \to \C)$ that is full and faithful.
\end{proof}

\section{Summary}

This article:

\begin{enumerate}
    \item describes a graphical calculus for mixed unitary categories,
    \item provides a description of completely positive maps in mixed unitary categories,
    \item generalizes the {\sf CP}-infinity construction from dagger monoidal categories to the more general setting of mixed unitary categories, and
    \item characterizes the new {\sf CP}-infinity construction for MUCs in terms of generalized environment structures. 
\end{enumerate}

Mixed unitary categories allow description of completely positive maps.  Moreover, the richer type system of mixed unitary categories allows a larger collection of completely positive maps when compared to dagger monoidal categories.  In particular, while all the completely positive maps must still factor through the self-adjoint core, positive maps can be between non-self-adjoint objects $(A \neq A^\dag)$.  

The power of the type system of MUCs also becomes evident when characterizing the {\sf CP}-infinity construction in terms of environment structures: in MUCs, one observes that discarding is available only for the self-adjoint (or unitary) objects. 

The objective of developing mixed unitary categories is to show that they can be used as a semantics of higher-order quantum programming languages and protocols which employ infinite dimensional structures: the {\sf CP}-infinity construction described in this article is a necessary step forward in this direction. 

\section*{Acknowledgement}
The authors would like to thank Cole Comfort and Jean Simon Lemay for many fruitful discussions and their insights on the project. This work was presented as talks at the 16th International Conference on Quantum Physics and Logic, Orange, CA and at the 2nd Symposium on Compositional Structures (SYCO 2), Glasgow. 

\bibliographystyle{plain}
\bibliography{cpinf-cons}

\appendix

\section{From Linearly distributive to monoidal categories}
\label{sec: Appendix A}

This section recalls the definition of mixed unitary categories introduced in \cite{CCS18} starting from linearly distributive categories \cite{CS97a}. 
 
 Linearly distributive categories (LDCs) are categories with two monoidal structures $(\ox, \top, a_\ox, u_\ox^l, u_\ox^r)$ and $(\oa, \bot, a_\oa, u_\oa^l, u_\oa^r)$ linked by natural transformations (which are not necessarily isomorphisms) called linear distributors:

\[
\partial^L: A \ox (B \oa C) \to (A \ox B) \oa C ~~~~~~~~~~~~~ \partial^R: (A \oa B) \ox C \to A \oa (B \ox C)
\]

A monoidal category is an LDC in which both the monoidal structures coincide, and the right and the left distributor gives the associator and its inverse respectively. LDCs provide the categorical semantics for multiplicative linear logic without negation. Moreover, LDCs are equipped with graphical calculus that subsumes the graphical calculus of the monoidal categories \cite{BCS00, Sch99}.  

One shall `compact' an LDC into a monoidal category in a series of steps as shown in Figure \ref{Fig: LDCs}. Along the way, one meets mix categories, isomix categories, and compact LDCs. 

		\begin{center}
		\begin{figure}[h]
			\centering
		\begin{tikzpicture}[scale=1.8]
			\begin{pgfonlayer}{nodelayer}
				\node [style=circle, scale=2, color=black, fill=red] (0) at (-5.75, 2.75) {};
				\node [style=circle, scale=2, color=black, fill=red!70] (1) at (-3.5, 2.75) {};
				\node [style=circle, scale=2, color=black, fill=red!60] (2) at (-1, 2.75) {};
				\node [style=circle, scale=2, color=black, fill=red!40] (3) at (1.75, 2.75) {};
				\node [style=circle, scale=2, color=black, fill=red!20] (5) at (4, 2.75) {};
				\node [style=none] (4) at (-7.75, 2.75) {};
				\node [style=none] (6) at (6, 2.75) {};
				\node [style=none] (7) at (-5.75, 2) {LDC};
				\node [style=none] (8) at (-3.5, 4) {Mix category};
				\node [style=none] (9) at (-3.5, 3.5) {$\m: \bot \to \top$};
				\node [style=none] (10) at (-1, 2) {Isomix category};
				\node [style=none] (11) at (1.75, 4.25) {Compact LDC};
				\node [style=none] (12) at (1.75, 3.65) {$A \ox B \to^{\mx}_{\simeq} A \oa B$};
				\node [style=none] (13) at (4, 2) {Monoidal category};
				\node [style=none] (14) at (-1, 1.4) {$\bot \to^{\m}_{\simeq} \top$};
				\node [style=none] (15) at (4, 1.5) {$\m = 1$, $\mx=1$};
				\node [style=none] (16) at (-5.75, 1.5) {$(\X, \ox, \top)$};
				\node [style=none] (17) at (-5.75, 1) {$(\X, \oa, \bot)$};
			\end{pgfonlayer}
			\begin{pgfonlayer}{edgelayer}
				\draw [dotted] (4.center) to (0);
				\draw (0) to (1);
				\draw (1) to (2);
				\draw (2) to (3);
				\draw (3) to (5);
				\draw [dotted] (5) to (6.center);
			\end{pgfonlayer}
		\end{tikzpicture}
		\caption{Schematic diagram of LDC properties}
		\label{Fig: LDCs}
	\end{figure}
\end{center}

A {\bf mix category} is an LDC $\X$ with a mix map $\m: \top \to \bot$ such that for any two objects $A, B \in \X$, the following coherence is satisfied, giving a map $\mx_{A,B}: A \ox B \to A \oa B$: 
\[
\xymatrixcolsep{4pc}
\xymatrix{
A \ox B \ar[r]^{1 \ox u_\oa^{L^{-1}}} \ar[d]_{(u_\oa^R)^{-1} \ox 1} \ar@{.>}[ddrr]^{\mx_{A,B}} & A \ox (\bot \oa B) 
\ar[r]^{1 \ox (\m \oa 1)} & A \ox ( \top \oa B) \ar[d]^{\partial^L} \\
(A \oa \bot) \ox B \ar[d]_{\partial^R} & & ( A \ox \top ) \oa B  \ar[d]^{u_\ox^R \oa 1} \\
A \oa (\bot \ox B) \ar[r]_{1 \oa (\m \ox 1)} & A \oa (\top \ox B) \ar[r]_{1 \oa u_\ox^L} &  A \oa B
}
\]
$\mx_{A,B}$ is called as the {\em mixor} and is natural in both the arguments. The commuting diagram is drawn as follows in the graphical calculus of LDCs:
\[
\mx_{A,B}:=
\begin{tikzpicture}
	\begin{pgfonlayer}{nodelayer}
		\node [style=otimes] (0) at (0, 0.2500001) {};
		\node [style=circ] (1) at (0.5000001, -0.2500001) {};
		\node [style=circ] (2) at (0, -1) {$\top$};
		\node [style=map] (3) at (0, -1.75) {m};
		\node [style=circ] (4) at (0, -2.5) {$\bot$};
		\node [style=circ] (5) at (-0.5000001, -3.25) {};
		\node [style=oplus] (6) at (0, -3.75) {};
		\node [style=nothing] (7) at (0, 0.7499999) {};
		\node [style=nothing] (8) at (0, -4.25) {};
	\end{pgfonlayer}
	\begin{pgfonlayer}{edgelayer}
		\draw (7) to (0);
		\draw (0) to (1);
		\draw [in=45, out=-60, looseness=1.00] (1) to (6);
		\draw [in=120, out=-135, looseness=1.00] (0) to (5);
		\draw (5) to (6);
		\draw (6) to (8);
		\draw [densely dotted, in=-90, out=45, looseness=1.00] (5) to (4);
		\draw (4) to (3);
		\draw (3) to (2);
		\draw [densely dotted, in=-135, out=90, looseness=1.00] (2) to (1);
	\end{pgfonlayer}
\end{tikzpicture}
=
\begin{tikzpicture}
	\begin{pgfonlayer}{nodelayer}
		\node [style=circ] (0) at (-0.5000001, -0.2500001) {};
		\node [style=circ] (1) at (0, -1) {$\top$};
		\node [style=map] (2) at (0, -1.75) {m};
		\node [style=circ] (3) at (0, -2.5) {$\bot$};
		\node [style=circ] (4) at (0.5000001, -3.25) {};
		\node [style=nothing] (5) at (0, 0.7499999) {};
		\node [style=nothing] (6) at (0, -4.25) {};
		\node [style=oplus] (7) at (0, -3.75) {};
		\node [style=otimes] (8) at (0, 0.2500001) {};
	\end{pgfonlayer}
	\begin{pgfonlayer}{edgelayer}
		\draw [densely dotted, in=-90, out=150, looseness=1.00] (4) to (3);
		\draw (3) to (2);
		\draw (2) to (1);
		\draw [densely dotted, in=-45, out=90, looseness=1.00] (1) to (0);
		\draw (8) to (5);
		\draw (8) to (0);
		\draw [in=135, out=-120, looseness=1.00] (0) to (7);
		\draw (7) to (6);
		\draw (7) to (4);
		\draw [in=-45, out=60, looseness=1.00] (4) to (8);
	\end{pgfonlayer}
\end{tikzpicture}
\]

The natural transformation $\mx_{A,B}$ is the {\bf mixor}.  When the map ${\sf m}$ is an isomorphism, then the LDC is said to be an {\bf isomix category};  furthermore,when  ${\sf m}$ is an isomorphism the coherence requirement above to obtain ${\sf mx}$ is automatic (see Lemma 6.6 \cite{CS97a}). A {\bf compact LDC} is an isomix LDC in which the mixor is a natural isomorphism. For a {\bf monoidal category}, the mixor is the identity natural transformation. 


\subsection{Linear duals}

A key notion in the theory of LDCs is the notion of a linear dual.   

Suppose $\mathbb{X}$ is a LDC and $A,B \in\X$, then $B$ is {\bf left linear dual}  (or left linear adjoint) to $A$ -- or $A$ is {\bf right linear dual} (right linear adjoint) to $B$ -- written $(\eta, \epsilon): B \dashv \!\!\!\! \dashv  A$, if there exists $\eta: \top \rightarrow B \oa A$ and $\epsilon: A \ox B \rightarrow \bot$ such that the following diagrams commute:

\[
\xymatrix{
B \ar[r]^{(u_\ox^L)^{-1}} \ar@{=}[d] 
& \top \ox B \ar[r]^{\eta \ox 1} 
& (B \oa A) \ox B \ar[d]^{\partial_R} \\
B 
& B \oa \bot \ar[l]^{u_\oa^R} 
& B \oa (A \ox B) \ar[l]^{1 \oa \epsilon}
}
~~~~~
\xymatrix{
A \ar[r]^{(u_\ox^R)^{-1}} \ar@{=}[d] 
& A \ox \top  \ar[r]^{1 \ox \eta} 
& A  \ox  (B \oa A)\ar[d]^{\partial_L} \\
A
& \bot \oa A \ar[l]^{u_\oa^L} 
& (A \ox B) \oa A   \ar[l]^{ \epsilon \oa 1} }
\]

An LDC in which every object has a chosen left and right dual is called a {\bf $*$-autonomous category}. In the above commuting diagrams, if $\X$ is a monoidal category, then the distributor is replaced by the associator, $\partial^R = a_\ox$ and $\ox = \oa$. In such a case, one gets a compact closed category.

\section{Linear functors and natural transformations}
\label{Sec: Appendix B}

Having defined $\dagger$-isomix categories,  we may now describe the appropirate functors between these categories. At a fundamental level, one would expect such functors to preserve the linear structure and the dagger. Functors between LDCs that preserve the linear structure are called linear functors. Given linearly distributive categories $\X$ and $\Y$, a linear functor $F: \X \to \Y$ consists of 
\begin{enumerate}[(i)]
\item a pair of functors $F = (F_\ox, F_\oa)$: $F_\ox$ which is monoidal with respect to $\ox$ and $F_\oa$ which is comonoidal with respect to $\oa$ 
\item natural transformations:
\begin{align*}
\nu_\ox^R &: F_\ox(A \oa B) \to F_\oa(A) \oa F_\ox(B) \\
\nu_\ox^L &: F_\ox(A \oa B) \to F_\ox(A) \oa F_\oa(B) \\
\nu_\oa^R &: F_\ox(A) \ox F_\oa(B) \to F_\oa(A \ox B) \\
\nu_\oa^L &: F_\oa(A) \ox F_\ox(B) \to F_\oa( A \ox B)
\end{align*}
\end{enumerate}
such that the certain coherence conditions hold. See \cite[Definition 1]{CS97} for the complete definition.


\begin{definition}
Suppose $\X$ and $\Y$ are LDCs. A {\bf Frobenius functor} is a linear functor $F: \X \to \Y$ such that:
\begin{enumerate}[{\bf \small [FLF.1]}]
\item $F_\ox = F_\oa $
\item $m_\ox = \nu_\oa^R = \nu_\oa^L $ 
\item $n_\oa = \nu_\ox^L = \nu_\ox^R$
\end{enumerate}
\end{definition}

Observe that it follows by definition that for any $\dagger$-LDC $\X$, $(\_)^\dagger: \X^{\op} \to \X$ is a Frobenius functor.

\begin{definition}
Suppose $\X$ and $\Y$ are mix categories. $F: \X \to \Y$ is a {\bf mix functor} if it is a Frobenius functor such that: 
\[
\mbox{\bf{[mix-FF]}}~~~~~
\xymatrix{
F(\bot) \ar@/_2pc/[rrr]_{F(\m)} \ar[r]^{n_\bot} & \bot \ar[r]^{\m} & \top \ar[r]^{m_\top} & F(\top) \\
}
\]
\end{definition}

This is diagrammatically represented using functor boxes as follows:  
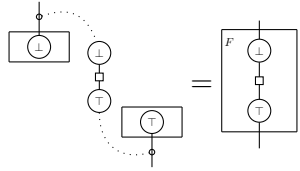
\[
\begin{tikzpicture}
	\begin{pgfonlayer}{nodelayer}
		\node [style=none] (0) at (-3, 2) {};
		\node [style=none] (1) at (-3, 1) {};
		\node [style=none] (2) at (-1, 2) {};
		\node [style=none] (3) at (-1, 1) {};
		\node [style=circle] (4) at (-2, 1.5) {$\bot$};
		\node [style=circle] (5) at (0, 1.3) {$\bot$};
		\node [style=none] (6) at (0, 0.5) {};
		\node [style=circle] (7) at (0, -0.3) {$\top$};
		\node [style=none] (8) at (-2, 3) {};
		\node [style=none] (9) at (0.75, -0.5) {};
		\node [style=none] (10) at (2.75, -0.5) {};
		\node [style=none] (11) at (0.75, -1.5) {};
		\node [style=none] (12) at (2.75, -1.5) {};
		\node [style=none] (13) at (1.75, -2.5) {};
		\node [style=circle] (14) at (1.75, -1) {$\top$};
		\node [style=map] (15) at (0, 0.5) {};
		\node [style=circle, scale=0.5] (16) at (1.75, -2) {};
		\node [style=circle, scale=0.5] (17) at (-2, 2.5) {};
	\end{pgfonlayer}
	\begin{pgfonlayer}{edgelayer}
		\draw (0.center) to (1.center);
		\draw (1.center) to (3.center);
		\draw (3.center) to (2.center);
		\draw (2.center) to (0.center);
		\draw (9.center) to (10.center);
		\draw (10.center) to (12.center);
		\draw (12.center) to (11.center);
		\draw (11.center) to (9.center);
		\draw (8.center) to (4);
		\draw (14) to (13.center);
		\draw [dotted, bend left=45, looseness=1.25] (17) to (5);
		\draw [dotted, in=-165, out=-90, looseness=1.25] (7) to (16);
		\draw (5) to (6.center);
		\draw (6.center) to (7);
	\end{pgfonlayer}
\end{tikzpicture}  = \begin{tikzpicture}
	\begin{pgfonlayer}{nodelayer}
		\node [style=circle] (0) at (0, 1.5) {$\bot$};
		\node [style=none] (1) at (0, 0.5) {};
		\node [style=circle] (2) at (0, -0.5) {$\top$};
		\node [style=map] (3) at (0, 0.5) {};
		\node [style=none] (4) at (0, -1.75) {};
		\node [style=none] (5) at (0, 2.5) {};
		\node [style=none] (6) at (-1.25, 2.2) {};
		\node [style=none] (7) at (-1.25, -1.2) {};
		\node [style=none] (8) at (1.25, 2.2) {};
		\node [style=none] (9) at (1.25, -1.2) {};
		\node [style=none] (10) at (-1, 1.8) {$F$};
	\end{pgfonlayer}
	\begin{pgfonlayer}{edgelayer}
		\draw (0) to (1.center);
		\draw (1.center) to (2);
		\draw (5.center) to (0);
		\draw (2) to (4.center);
		\draw (6.center) to (8.center);
		\draw (8.center) to (9.center);
		\draw (9.center) to (7.center);
		\draw (7.center) to (6.center);
	\end{pgfonlayer}
\end{tikzpicture} \]

\begin{lemma}
\label{Lemma: Mix Frobenius linear functor}
Mix functors preserve the mix map:
\[
\xymatrix{
F(A) \ox F(B) \ar[r]^{\mx} \ar[d]_{m_\ox} \ar[r]^{\mx} & F(A) \oa F(B) \\
F(A \ox B) \ar[r]_{F(\mx)} \ar[r]_{F(\mx)} & F(A \oa B) \ar[u]_{n_\oa}
}
\]
\end{lemma}

\begin{definition}
A Frobenius functor between isomix categories is an {\bf isomix functor} in case it is a mix functor which satisfies, in addition, the following diagram:
\[ \mbox{\bf{[isomix-FF]}}~~~~~\xymatrix{ \top \ar@/^/[rrr]^{{\sf m}^{-1}} \ar[dr]_{m_\top} & & & \bot \\ & F(\top) \ar[r]_{F({\sf m}^{-1})} & F(\bot) \ar[ur]_{n_\top} } \]
\end{definition}

Note that when $\X$ is a $\dagger$-isomix category, $(\_)^\dagger: \X^{\op} \to \X$ is an isomix functor.

\end{document}